\documentclass[accepted]{uai2025} 
                        
\usepackage[american]{babel}

\usepackage{natbib} 
\usepackage{mathtools} 
\usepackage{booktabs} 
\usepackage{tikz} 
\usepackage{amssymb}

\usepackage{amsmath,amsthm,amssymb,color,enumitem,tikz,float,subcaption,algpseudocode,algorithm}
\usepackage{multirow}

\newcounter{phase}[algorithm]
\newlength{\phaserulewidth}

\makeatother

\newcommand{\EE}{\mathbb{E}}

\newcommand{\rT}{{\rm T}}
\newcommand{\cC}{{\cal C}}

\newcommand{\cP}{{\cal P}}
\newcommand{\cT}{{\cal T}}

\newcommand{\Top}{{\rm top}}
\newcommand{\pa}{{\rm pa}}
\newcommand{\an}{{\rm anc}}

\newcommand{\eps}{\varepsilon}

\newcommand{\proofof}[1]{\noindent {\sc Proof of #1.} \quad}

\newtheorem{thm}{Theorem}[section]
\newtheorem{rmk}[thm]{Remark}
\newtheorem{lem}[thm]{Lemma}
\newtheorem{cor}[thm]{Corollary}
\newtheorem{prop}[thm]{Proposition}

\newtheorem{defn}[thm]{Definition}
\newtheorem{exmp}[thm]{Example}

\newcommand{\rf}[1]{(\ref{#1})}
\newcommand{\trf}[1]{Theorem~\ref{#1}}

\newcommand{\erf}[1]{Example~\ref{#1}}

\title{Causal Discovery for Linear Non-Gaussian \\ Models with Disjoint Cycles}

\author[1]{\href{mailto:<mathias.drton@tum.de>}{Mathias Drton}{}}
\author[2]{\href{mailto:<marinagl@kth.se>}{Marina Garrote-L\'opez}}
\author[3]{\href{mailto:<niko.a.nikov@gmail.com>}{Niko Nikov}}
\author[4]{\href{mailto:<erobeva@math.ubc.ca>}{Elina Robeva}}
\author[5]{\href{mailto:<ysw7@cornell.edu>}{Y. Samuel Wang}}
\affil[1]{%
    Technical University of Munich and Munich Center for Machine Learning\\
    Munich, Germany
}
\affil[2]{%
    KTH Royal Institute of Technology\\
    Stockholm, Sweden
}
\affil[3]{%
    University of Edinburgh\\
    Edinburgh, UK
  }
  \affil[4]{%
University of British Columbia\\
Vancouver, BC, Canada
  }
  \affil[5]{%
Cornell University\\
Ithaca, NY, USA
  }
  
\begin{document}
\maketitle

\begin{abstract}
  The paradigm of linear structural equation modeling readily allows
  one to incorporate causal feedback loops in the model specification.
  These appear as directed cycles in the common graphical
  representation of the models.  However, the presence of cycles
  entails difficulties such as the fact that models need no
  longer be characterized by conditional independence relations.  As
  a result, learning cyclic causal structures  remains a challenging
  problem.  In this paper, we offer new insights on this problem in
  the context of linear non-Gaussian models. First, we precisely
  characterize when two directed graphs determine the same
  linear non-Gaussian model. Next, we take up a setting of cycle-disjoint graphs, for which we are able to show that simple quadratic and cubic polynomial relations among low-order moments of a non-Gaussian distribution allow one to locate source cycles. Complementing this with a strategy of decorrelating cycles and multivariate regression allows one to infer a block-topological order among the directed cycles, which leads to a {consistent and computationally efficient algorithm} for learning causal structures with disjoint cycles.
\end{abstract}

\section{Introduction}
In this work we study linear structural causal models~\citep{pearl:2013} that allow for feedback loops.
Throughout, we consider a collection of observed random variables $(X_i)_{i\in V}$ that is indexed by the nodes of a directed graph $G=(V,E)$. The {\it linear structural equation model} (LSEM) associated to the graph $G$ hypothesizes that the random variables $(X_i)_{i\in V}$ satisfy a system of structural equations
\begin{align}\label{eq:linear_structural_equations}
X_i = \sum_{j\in\pa(i)}\lambda_{ji}X_j+\eps_i,
\end{align}
where $\pa(i)=\{j\in V:j\to i\in E\}$ is the set of parents of node
$i\in V$, and the noise terms $(\eps_i)_{i\in V}$ are centered and
independent.  Our problem of interest is causal discovery from
observational data, i.e., given i.i.d.~samples from this model, we are
interested in learning the graph $G$  and the unknown coefficients
$\lambda_{ij}\in\mathbb{R}$.

Our focus is on discovery in models with non-Gaussian errors.
The non-Gaussianity may render causal directions identifiable also in
settings where classical Markov equivalence theory prevents
identification \citep{maathuis2018handbook}. In the
well understood setting of directed acyclic graphs
(DAGs), each linear non-Gaussian model is associated to a unique DAG
\citep{shimizu2006linear}.  This fact has spurred active development
of 
discovery algorithms
\citep{shimizu2011directlingam}, also dealing with issues of high dimensionality \citep{wang2019highdimensional,DBLP:conf/uai/TramontanoMD22} or latent variables \citep{salehkaleybar:2020,liu:2021,DBLP:conf/nips/AdamsHZ21,cai:2023,wang:2023, schkoda:2024}.  For a comprehensive review beyond these selective references, see \cite{shimizu:2022}.

A problem that remains challenging also in the linear non-Gaussian case, however, is discovery of causal structures that are allowed to contain directed cycles.  Directed cycles encode feedback loops and lead to difficulties such as models that cannot be described in terms of conditional independence and whose linear causal effects do not correspond to coefficients in conditional expectations \citep{spirtes:using_path_diagrams}.  For linear Gaussian models, this results in the fact that conditional independence relations alone do not allow one to learn precisely which variables are adjacent to each other in the graph  \citep{DBLP:conf/uai/Richardson96a,ghassamiYKZ:2020}.  Hence, while algorithms focusing on conditional independence tests \citep{DBLP:conf/uai/Richardson96, Semnani2025} can be useful, they are unable to infer all identifiable aspects of the underlying causal structure.


Work of \cite{DBLP:conf/uai/LacerdaSRH08} stresses the possibility of different graphs with reversed cycles inducing the same linear non-Gaussian models.  In the present  work, we revisit this issue and give a complete characterization of when any two directed graphs are equivalent in the linear non-Gaussian setting (see Section~\ref{sec:equivalence}). 

\cite{DBLP:conf/uai/LacerdaSRH08} and more recently~\citet{dai:2024} show that cyclic models can be learned using independent component analysis (ICA) or Independent Subspace Analysis (ISA) combined with a sorting strategy over all permutations of the considered variables.  Here we consider a setting that lies between the classical DAGs and the fully arbitrary structures considered by \cite{DBLP:conf/uai/LacerdaSRH08}.  
Specifically, we focus on graphs in which cycles remain disjoint {(i.e., each node belongs to at most one cycle)}. While such graphs do appear in applications \citep{grace2016integrative}, our motivation is computational and focused on a discovery algorithm that learns such graphs computationally
efficiently (Section~\ref{sec:algorithm}).  In particular, our proposed
algorithm avoids the search over permutations that features in
ICA-based approaches. 

We remark that the algorithm
of~\citet{DBLP:conf/uai/LacerdaSRH08}  outputs all graphs and edge
weights for which the product of absolute edge weights along any cycle
is smaller than 1. This can also be done in our setting of cycle-disjoint graphs.

A key insight for our algorithm is that the strong components of a
\emph{cycle-disjoint graph} can be learned 
using simple quadratic and cubic constraints on the second
and third moments.  The algorithm 
uses the constraints to recursively identify initial strong components
and leverages further algebraic relations among moments to learn
the  edges inside each component.  Multivariate regression then allows
one to determine the edges between the components.  Our procedure is
reminiscent of that of~\cite{DBLP:journals/corr/abs-1006-5041}, which
learns a DAG over groups of variables but does not offer an efficient way to find the right grouping of the variables.

The rest  of the paper is organized as follows. Section~\ref{sec:preliminaries} reviews preliminaries for linear non-Gaussian SEMs.  Section~\ref{sec:equivalence} provides a complete characterization of the equivalence classes of directed graphs  which give rise to the same linear non-Gaussian model (Thm.~\ref{distrequivalence}). In Section~\ref{sec:algorithm}, we propose our algorithm for causal discovery based on second and third 
 moments of the observed random vector.  Our method works for
 cycle-disjoint directed graphs, i.e., directed graphs in which no two
 directed cycles intersect (Def.~\ref{def:cycledisjoint}). In such graphs, if several vertices form a directed cycle, then there are no additional edges among those vertices. We conclude with Section~\ref{sec:numerical_experiments}, which presents  numerical experiments.

\section{Preliminaries}\label{sec:preliminaries}

Let $G=(V,E)$ be a directed graph (DG) with vertices $V$ and directed edges $E$. Since we allow directed cycles in our graph, it is not always possible to find a topological ordering of the vertices (an ordering so that edges only go from a smaller to a larger numbered vertex). Instead, there exists such an ordering among the strong components of $G$.

\begin{defn}
    A subset of vertices $C\subseteq V$ is a {\em strong component}
    (or a {\em strongly connected component}) of $G$, if there exists a directed path from any vertex $i\in C$ to any other vertex $j\in C$, and $C$ is a maximal such subset.
\end{defn}

Now, let $X = (X_i)_{i\in V}$ be a random vector that satisfies the linear structural equations of~\eqref{eq:linear_structural_equations}. We can rewrite the equations in~\eqref{eq:linear_structural_equations} in matrix-vector format as
\begin{equation}
\label{eq:linear_structural_equations:vector}
 X = \Lambda^\rT X+\eps,   
\end{equation} 
where $\epsilon = (\epsilon_i)_{i\in  V}$ collects the noise terms,
and the
parameter matrix
$\Lambda = (\lambda_{ij})_{i,j\in V}$ has $\lambda_{ij}=0$ unless $i\to j$ is an edge in $E$.
For a well-defined model, the matrix $\Lambda$ should satisfy $I-\Lambda$ being invertible.  Indeed, precisely then, the system
in~\eqref{eq:linear_structural_equations:vector} admits a unique
solution, $X=(I-\Lambda)^{-\rT}\eps$.

\begin{defn}\label{def:Lambda}
 We denote the set of $|V|\times |V|$ matrices $\Lambda=(\lambda_{ij})$ for which $I-\Lambda$ is invertible and $\lambda_{ij} = 0$ if $i\to j\not\in E$ by $\mathbb R^{E}$. Denote the set of $\Lambda$ which in addition satisfy $\lambda_{ij}\neq 0$ if $i\to j \in E$ by $\mathbb R^{E}_*$.
\end{defn}

We assume throughout that the random vector $\varepsilon$, and thus also $X=(I-\Lambda)^{-\rT}\eps$, is centered. The  second moments of $X$ form the covariance matrix
\begin{equation}
    \label{eq:S}
    S=(s_{ij})_{i,j\in V}=(\EE[X_iX_j])_{i,j\in V}.
\end{equation}
Similarly, the third-order moments of $X$ form a tensor of order 3, which we denote by
\begin{equation}
    \label{eq:T}
T=(t_{ijk})_{i,j,k\in V}=(\EE[X_iX_jX_k])_{i,j,k\in V}. 
\end{equation}
Let $\Omega^{(2)}=(\EE[\eps_{i}\eps_{j}])_{i,j\in V}$ and $\Omega^{(3)}=(\EE[\eps_{i}\eps_{j}\eps_{k}])_{i,j,k\in V}$
be the second and third-order moment tensors  of $\varepsilon$. By the independence of the $\varepsilon_i$, both $\Omega^{(2)}$ and $\Omega^{(3)}$ are diagonal tensors whose diagonal entries we denote by $\omega^{(2)}_i$ and $\omega^{(3)}_i$, respectively. Since $X=(I-\Lambda)^{-\rT}\eps$, we have that
\begin{align*}
  S = S(\Lambda,\Omega^{(2)}) &:= (I-\Lambda)^{-\rT}\Omega^{(2)}(I-\Lambda)^{-1}, \\
  T = T(\Lambda,\Omega^{(3)}) & \\
  :=\Omega^{(3)}&\bullet(I-\Lambda)^{-1}\bullet(I-\Lambda)^{-1}\bullet(I-\Lambda)^{-1},
\end{align*}
where $\bullet$ denotes \emph{Tucker product}.  In other words, the entry $t_{ijk}$ of $T$ equals
\begin{align}
    \label{eq:T-tucker}
    &\sum_{a,b,c\in V}[(I- \Lambda)^{-1}]_{ai}[(I- \Lambda)^{-1}]_{bj}
    [(I- \Lambda)^{-1}]_{ck} \Omega^{(3)}_{abc}\\
        \label{eq:T-tucker-par}
&= \sum_{a\in V}[(I- \Lambda)^{-1}]_{ai}[(I- \Lambda)^{-1}]_{aj}
    [(I- \Lambda)^{-1}]_{ak} \Omega^{(3)}_{aaa}.
\end{align}
Going forward, we will study the non-Gaussian model
specified by~\eqref{eq:linear_structural_equations:vector} via the set of second- and third-order moments
\begin{multline*}
  \mathcal{M}^{2,3}(G)=\bigg\{ \big(S(\Lambda,\Omega^{(2)}),T(\Lambda,\Omega^{(3)})\big): \Lambda\in\mathbb{R}^{E}, \\
   \Omega^{(2)}\in \mathbb R^{|V|\times|V|} \text{ diagonal and positive definite},\;\\
   \Omega^{(3)}\in \mathbb R^{|V|\times|V|\times|V|} \text{ diagonal} \bigg\}.
   \end{multline*}
  We call this set the {\em model of second- and third-order moments} corresponding to $G$.

  While the covariance matrix $S$ captures the entire dependence structure in the Gaussian case, higher-order moments are relevant in the linear non-Gaussian setting. Assuming that the diagonal elements of $\Omega^{(3)}$ are nonzero (that is, the distribution of $\varepsilon$ is skewed), the second- and third-order moments are enough to identify the graph $G$ and the coefficient matrix $\Lambda$
  when $G$ is acyclic \citep[e.g.,][]{wang2019highdimensional}. In Section~\ref{sec:algorithm}, we use the second- and third-order moments to learn  graphs we term cycle-disjoint (Def.~\ref{def:cycledisjoint}).

The classical notion of treks describes the source of nonzero
covariances in linear SEMs. In our linear non-Gaussian setting, the more general notion of a $k$-trek is needed.

\begin{defn} \label{multitrek}
    A $k$-trek between the vertices $i_1,\ldots,i_k$ is a $k$-tuple of paths $\tau=(P_1,\ldots,P_k)$ with common source $\ell=\Top(\tau)$ and respective sinks $i_1,\ldots,i_k$.  The node $\ell$ is called the top of the trek. The set of all $k$-treks between $i_1,\ldots,i_k$ is denoted by $\cT(i_1,\ldots,i_k)$.
 {A trek $\tau = (P_1,\ldots,P_k) \in \cT(i_1,\ldots,i_k)$ is called a \emph{simple trek} if $P_u \cap P_v = \text{top}(\tau)$ for all $u,v \in [k]$, that is, the only common vertex between each pair of paths $P_u$ and $P_v$ is the top of the trek. Moreover, we say that $\tau$ has \emph{non-empty sides} if any path $P_u$ in $\tau$ is non-empty, i.e., $\text{top}(\tau) \neq i_u$.}
   
\end{defn}

{The {\em trek rule} uses these treks to specify how the entries of $S$ and $T$ can be expressed in terms of the entries of $\Lambda$, $\Omega^{(2)}$, and  $\Omega^{(3)}$. It was originally used in the Gaussian setting~\citep{sullivant2010trek} and was later extended to higher-order moments for linear non-Gaussian models~\citep{robeva2021multi-trek, amendola2023third}.
It states that the covariances $s_{ij}$ and third moments $t_{ijk}$ can be expressed as
\begin{align}\label{eq:trekrule}
        &s_{ij}=\sum_{(P_i,P_j)\in\cT(i,j)}\lambda^{P_i}\lambda^{P_j}\omega^{(2)}_{\Top(P_i,P_j)}, \\
        &t_{ijk}=\sum_{(P_i,P_j,P_k)\in\cT(i,j,k)}\lambda^{P_i}\lambda^{P_j}\lambda^{P_k}\omega^{(3)}_{\Top(P_i,P_j,P_k)}, \nonumber
    \end{align}
    where for a path $P=(u_1,\ldots,u_r)$ we write $\lambda^P=\lambda_{u_1u_2}\lambda_{u_2u_3}\ldots\lambda_{u_{r-1}u_r}$.}

\section{Distribution Equivalent Graphs}\label{sec:equivalence}

In the case of acyclic graphs, observational data is enough to
uniquely recover the underlying causal
graph~\citep{shimizu2006linear}. However, when we allow graphs
containing cycles, we cannot identify the graph (and thus also
parameters) uniquely.
We now state a characterization of all the graphs and parameters that give rise to the same distributions. {The discussion and results of this section hold for all directed graphs and do not require the disjoint cycles assumption we make later when discussing discovery.}

For any directed graph $G=(V,E)$, let $\mathcal{P}(G)$ be the set of all joint distributions of random vectors $X=(I-\Lambda)^{-1}\eps$ that solve the structural equations from \eqref{eq:linear_structural_equations:vector} for $\Lambda\in\mathbb R^E_*$ and $\eps$ a random noise vector with independent components.

\begin{defn}\label{def:equivalence}
  Let $G = (V, E)$ and $G'  = (V,  E')$ be two graphs on the same set of vertices.  Then $G$ and $G'$ are  {\em distribution equivalent} if $\mathcal{P}(G)=\mathcal{P}(G')$.
\end{defn}

The following~\erf{equivalenceexmp1} shows that we can obtain
distribution equivalent graphs by transforming the structural equations. We then
prove in \trf{distrequivalence} that in fact all distribution equivalences arise
from such transformations. Our consideration of nonzero parameters $\lambda_{ij}$ in
Definition~\ref{def:Lambda} is tied to the fact that to transform the models,  we need to divide by the
$\lambda_{ij}$ coefficients; see also Example~\ref{ex:zero_lambdas} below.

\begin{exmp} \label{equivalenceexmp1}
  Let $G$ and $G'$ be the DGs pictured in Figure~\ref{fig:EquivalentGraphs}. The graph $G$ gives rise to the following structural equations in variables $3,4,5$:
  $$ \begin{cases} X_3 = \lambda_{13}X_1+\lambda_{53}X_5+\eps_3, \\ X_4 = \lambda_{24}X_2+\lambda_{34}X_3+\eps_4, \\ X_5 =\lambda_{45}X_4+\eps_5. \end{cases} $$

\begin{figure} 
    \centering
    
    \tikzset{every picture/.style={line width=0.6pt}} 
    
    \begin{tikzpicture}[x=0.75pt,y=0.75pt,yscale=-0.8,xscale=0.8]
    
    \draw    (125,154.67) -- (98.69,138.07) ;
    \draw [shift={(97,137)}, rotate = 32.25] [color={rgb, 255:red, 0; green, 0; blue, 0 }  ][line width=0.75]    (10.93,-3.29) .. controls (6.95,-1.4) and (3.31,-0.3) .. (0,0) .. controls (3.31,0.3) and (6.95,1.4) .. (10.93,3.29)   ;
    \draw    (142,99.67) -- (171.33,99.67) ;
    \draw [shift={(173.33,99.67)}, rotate = 180] [color={rgb, 255:red, 0; green, 0; blue, 0 }  ][line width=0.75]    (10.93,-3.29) .. controls (6.95,-1.4) and (3.31,-0.3) .. (0,0) .. controls (3.31,0.3) and (6.95,1.4) .. (10.93,3.29)   ;
    \draw    (100,77) -- (125.64,92.94) ;
    \draw [shift={(127.33,94)}, rotate = 211.88] [color={rgb, 255:red, 0; green, 0; blue, 0 }  ][line width=0.75]    (10.93,-3.29) .. controls (6.95,-1.4) and (3.31,-0.3) .. (0,0) .. controls (3.31,0.3) and (6.95,1.4) .. (10.93,3.29)   ;
    \draw    (127.33,105.67) -- (102.39,120.32) ;
    \draw [shift={(100.67,121.33)}, rotate = 329.57] [color={rgb, 255:red, 0; green, 0; blue, 0 }  ][line width=0.75]    (10.93,-3.29) .. controls (6.95,-1.4) and (3.31,-0.3) .. (0,0) .. controls (3.31,0.3) and (6.95,1.4) .. (10.93,3.29)   ;
    \draw    (228,114.67) -- (228,84.67) ;
    \draw [shift={(228,82.67)}, rotate = 90] [color={rgb, 255:red, 0; green, 0; blue, 0 }  ][line width=0.75]    (10.93,-3.29) .. controls (6.95,-1.4) and (3.31,-0.3) .. (0,0) .. controls (3.31,0.3) and (6.95,1.4) .. (10.93,3.29)   ;
    \draw    (192.33,108.33) -- (217.97,124.28) ;
    \draw [shift={(219.67,125.33)}, rotate = 211.88] [color={rgb, 255:red, 0; green, 0; blue, 0 }  ][line width=0.75]    (10.93,-3.29) .. controls (6.95,-1.4) and (3.31,-0.3) .. (0,0) .. controls (3.31,0.3) and (6.95,1.4) .. (10.93,3.29)   ;
    \draw    (219.67,75.33) -- (193.39,90.98) ;
    \draw [shift={(191.67,92)}, rotate = 329.24] [color={rgb, 255:red, 0; green, 0; blue, 0 }  ][line width=0.75]    (10.93,-3.29) .. controls (6.95,-1.4) and (3.31,-0.3) .. (0,0) .. controls (3.31,0.3) and (6.95,1.4) .. (10.93,3.29)   ;
    \draw    (288.67,29.33) .. controls (270.28,49.69) and (267.09,99.48) .. (286.12,127.72) ;
    \draw [shift={(287,129)}, rotate = 234.46] [color={rgb, 255:red, 0; green, 0; blue, 0 }  ][line width=0.75]    (10.93,-3.29) .. controls (6.95,-1.4) and (3.31,-0.3) .. (0,0) .. controls (3.31,0.3) and (6.95,1.4) .. (10.93,3.29)   ;
    \draw    (89.67,118) -- (89.67,86.33) ;
    \draw [shift={(89.67,84.33)}, rotate = 90] [color={rgb, 255:red, 0; green, 0; blue, 0 }  ][line width=0.75]    (10.93,-3.29) .. controls (6.95,-1.4) and (3.31,-0.3) .. (0,0) .. controls (3.31,0.3) and (6.95,1.4) .. (10.93,3.29)   ;
    \draw    (340.33,142) -- (340.96,112.67) ;
    \draw [shift={(341,110.67)}, rotate = 91.22] [color={rgb, 255:red, 0; green, 0; blue, 0 }  ][line width=0.75]    (10.93,-3.29) .. controls (6.95,-1.4) and (3.31,-0.3) .. (0,0) .. controls (3.31,0.3) and (6.95,1.4) .. (10.93,3.29)   ;
    \draw    (346.67,99.67) -- (376,99.67) ;
    \draw [shift={(378,99.67)}, rotate = 180] [color={rgb, 255:red, 0; green, 0; blue, 0 }  ][line width=0.75]    (10.93,-3.29) .. controls (6.95,-1.4) and (3.31,-0.3) .. (0,0) .. controls (3.31,0.3) and (6.95,1.4) .. (10.93,3.29)   ;
    \draw    (306.36,78.06) -- (332,94) ;
    \draw [shift={(304.67,77)}, rotate = 31.88] [color={rgb, 255:red, 0; green, 0; blue, 0 }  ][line width=0.75]    (10.93,-3.29) .. controls (6.95,-1.4) and (3.31,-0.3) .. (0,0) .. controls (3.31,0.3) and (6.95,1.4) .. (10.93,3.29)   ;
    \draw    (330.28,106.68) -- (305.33,121.33) ;
    \draw [shift={(332,105.67)}, rotate = 149.57] [color={rgb, 255:red, 0; green, 0; blue, 0 }  ][line width=0.75]    (10.93,-3.29) .. controls (6.95,-1.4) and (3.31,-0.3) .. (0,0) .. controls (3.31,0.3) and (6.95,1.4) .. (10.93,3.29)   ;
    \draw    (294.33,116) -- (294.33,84.33) ;
    \draw [shift={(294.33,118)}, rotate = 270] [color={rgb, 255:red, 0; green, 0; blue, 0 }  ][line width=0.75]    (10.93,-3.29) .. controls (6.95,-1.4) and (3.31,-0.3) .. (0,0) .. controls (3.31,0.3) and (6.95,1.4) .. (10.93,3.29)   ;
    \draw    (89.67,39.83) -- (89.06,60.83) ;
    \draw [shift={(89,62.83)}, rotate = 271.66] [color={rgb, 255:red, 0; green, 0; blue, 0 }  ][line width=0.75]    (10.93,-3.29) .. controls (6.95,-1.4) and (3.31,-0.3) .. (0,0) .. controls (3.31,0.3) and (6.95,1.4) .. (10.93,3.29)   ;
    \draw    (433,114.67) -- (433,84.67) ;
    \draw [shift={(433,82.67)}, rotate = 90] [color={rgb, 255:red, 0; green, 0; blue, 0 }  ][line width=0.75]    (10.93,-3.29) .. controls (6.95,-1.4) and (3.31,-0.3) .. (0,0) .. controls (3.31,0.3) and (6.95,1.4) .. (10.93,3.29)   ;
    \draw    (397.33,108.33) -- (422.97,124.28) ;
    \draw [shift={(424.67,125.33)}, rotate = 211.88] [color={rgb, 255:red, 0; green, 0; blue, 0 }  ][line width=0.75]    (10.93,-3.29) .. controls (6.95,-1.4) and (3.31,-0.3) .. (0,0) .. controls (3.31,0.3) and (6.95,1.4) .. (10.93,3.29)   ;
    \draw    (424.67,75.33) -- (398.39,90.98) ;
    \draw [shift={(396.67,92)}, rotate = 329.24] [color={rgb, 255:red, 0; green, 0; blue, 0 }  ][line width=0.75]    (10.93,-3.29) .. controls (6.95,-1.4) and (3.31,-0.3) .. (0,0) .. controls (3.31,0.3) and (6.95,1.4) .. (10.93,3.29)   ;
    
    \draw (130.67,149.23) node [anchor=north west][inner sep=0.75pt]  [font=\footnotesize]  {$1$};
    \draw (85.5,123.23) node [anchor=north west][inner sep=0.75pt]  [font=\footnotesize]  {$3$};
    \draw (85.5,64.57) node [anchor=north west][inner sep=0.75pt]  [font=\footnotesize]  {$4$};
    \draw (223.5,64.57) node [anchor=north west][inner sep=0.75pt]  [font=\footnotesize]  {$8$};
    \draw (131.67,93.07) node [anchor=north west][inner sep=0.75pt]  [font=\footnotesize]  {$5$};
    \draw (178,92.57) node [anchor=north west][inner sep=0.75pt]  [font=\footnotesize]  {$6$};
    \draw (223.5,123.23) node [anchor=north west][inner sep=0.75pt]  [font=\footnotesize]  {$7$};
    \draw (335.33,147.9) node [anchor=north west][inner sep=0.75pt]  [font=\footnotesize]  {$1$};
    \draw (292.67,21.57) node [anchor=north west][inner sep=0.75pt]  [font=\footnotesize]  {$2$};
    \draw (290.17,123.23) node [anchor=north west][inner sep=0.75pt]  [font=\footnotesize]  {$3$};
    \draw (290.17,64.57) node [anchor=north west][inner sep=0.75pt]  [font=\footnotesize]  {$4$};
    \draw (336.33,93.07) node [anchor=north west][inner sep=0.75pt]  [font=\footnotesize]  {$5$};
    \draw (382.67,92.57) node [anchor=north west][inner sep=0.75pt]  [font=\footnotesize]  {$6$};
    \draw (85.33,22.73) node [anchor=north west][inner sep=0.75pt]  [font=\footnotesize]  {$2$};
    \draw (148.67,19.57) node [anchor=north west][inner sep=0.75pt]  [font=\small]  {$G$};
    \draw (354.67,19.57) node [anchor=north west][inner sep=0.75pt]  [font=\small]  {$G'$};
    \draw (428.5,64.57) node [anchor=north west][inner sep=0.75pt]  [font=\footnotesize]  {$8$};
    \draw (428.5,123.23) node [anchor=north west][inner sep=0.75pt]  [font=\footnotesize]  {$7$};

    \end{tikzpicture}

     \caption{Distribution equivalent graphs $G$ and $G'$.}
     \label{fig:EquivalentGraphs}
\end{figure}
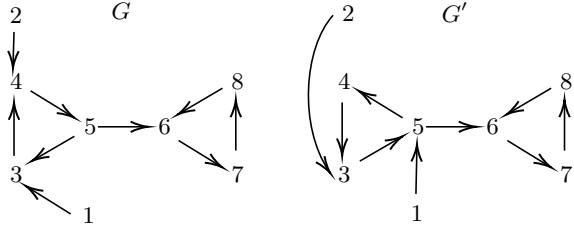

  Assuming all of the $\lambda$s are nonzero, we can rearrange these equations to solve for the variables on the right-hand side of each equation and obtain
  $$ \begin{cases} 
  X_3 = \underbrace{(-\lambda_{24}/\lambda_{34})}_{\lambda_{23}'}X_2+\underbrace{(1/\lambda_{34})}_{\lambda_{43}'}X_4+\underbrace{(-1/\lambda_{34})\eps_4}_{\eps_3'},\\ 
  X_4 = \underbrace{(1/\lambda_{45})}_{\lambda_{54}'}X_5+\underbrace{(-1/\lambda_{45})\eps_5}_{\eps_4'},\\
  X_5 = \underbrace{(-\lambda_{13}/\lambda_{53})}_{\lambda_{15}'}X_1+\underbrace{(1/\lambda_{53})}_{\lambda_{35}'}X_3+\underbrace{(-1/\lambda_{53})\eps_3}_{\eps_5'}
  \end{cases} $$
  Note that this transformation does not affect the variables $1,2,6,7,8$. In particular, $X$ also lies in the model generated by $G'$. Therefore, $G$ and $G'$ are distribution equivalent.
\end{exmp}

\begin{defn}
  Let $G=(V,E)$ be a DG, and let $\pi$ be a permutation of $V$. We say that $\pi$ {\em factors according to vertex-disjoint cycles} in $G$ if for all $\tau=(i_1\ i_2\ \cdots\ i_s)$ in the disjoint cycle decomposition of $\pi$, the directed cycle $i_1\to i_2\to\cdots\to i_s\to i_1$ is in $G$. Such a permutation $\pi$ determines a unique DG $\pi(G)=(V,\pi(E))$ via
  \begin{enumerate}[label=(\roman*)]
  \item if $i=\pi(j)$ {and $i\neq j$}, then $i\to j\in \pi(E)$, and
  \item if $i\ne \pi(j)$ {and $i\neq j$}, then $i\to j\in \pi(E)$ if and only if $i\to \pi(j)\in E$.
  \end{enumerate}
  Condition (ii) above is equivalent to ``if $i\ne \pi(j)$, then $i\to \pi^{-1}(j)\in\pi(E)$ if and only if $i\to j\in E$.''
\end{defn}

Put differently, in order to construct the graph $\pi(G)$, we need to reverse one or several disjoint cycles in $G$, and if the edge $i\to j$ was on one of these cycles (and thus got reversed to $j\to i$), all edges $k\to j$ that originally pointed to $j$ have to now be changed to point to $i$, i.e., $k\to i$.

\noindent Note that in \erf{equivalenceexmp1} we had $G'=\pi(G)$ with $\pi=(3\ 4\ 5)$. In other words, we reversed the cycle $3\to4\to5\to3$, and we had to adjust the edges pointing to $3,4,5$ according to the rule above. Although in this example all cycles in $G$ are vertex-disjoint, this need not be the case in general. Here we provide another example.

\begin{exmp} \label{equivalenceexmp2}
  Figure~\ref{fig:non-cycle-disjoint}  pictures  $G$ and $\pi(G)$ for two different choices of $\pi$. In fact,  these three graphs comprise an entire distribution equivalence class.

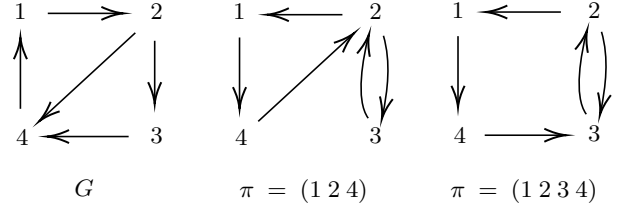
\begin{figure} 
    \centering

    \tikzset{every picture/.style={line width=0.6pt}} 
    
    \begin{tikzpicture}[x=0.75pt,y=0.75pt,yscale=-1,xscale=1]

    \draw    (108.67,71.33) -- (108.67,101.17) ;
    \draw [shift={(108.67,103.17)}, rotate = 270] [color={rgb, 255:red, 0; green, 0; blue, 0 }  ][line width=0.75]    (10.93,-3.29) .. controls (6.95,-1.4) and (3.31,-0.3) .. (0,0) .. controls (3.31,0.3) and (6.95,1.4) .. (10.93,3.29)   ;
    \draw    (55.33,57.5) -- (91,57.5) ;
    \draw [shift={(93,57.5)}, rotate = 180] [color={rgb, 255:red, 0; green, 0; blue, 0 }  ][line width=0.75]    (10.93,-3.29) .. controls (6.95,-1.4) and (3.31,-0.3) .. (0,0) .. controls (3.31,0.3) and (6.95,1.4) .. (10.93,3.29)   ;
    \draw    (99,67.33) -- (52.46,110.64) ;
    \draw [shift={(51,112)}, rotate = 317.06] [color={rgb, 255:red, 0; green, 0; blue, 0 }  ][line width=0.75]    (10.93,-3.29) .. controls (6.95,-1.4) and (3.31,-0.3) .. (0,0) .. controls (3.31,0.3) and (6.95,1.4) .. (10.93,3.29)   ;
    \draw    (41.67,105.67) -- (41.67,71.67) ;
    \draw [shift={(41.67,69.67)}, rotate = 90] [color={rgb, 255:red, 0; green, 0; blue, 0 }  ][line width=0.75]    (10.93,-3.29) .. controls (6.95,-1.4) and (3.31,-0.3) .. (0,0) .. controls (3.31,0.3) and (6.95,1.4) .. (10.93,3.29)   ;
    \draw    (95.67,119.67) -- (57,119.67) ;
    \draw [shift={(55,119.67)}, rotate = 360] [color={rgb, 255:red, 0; green, 0; blue, 0 }  ][line width=0.75]    (10.93,-3.29) .. controls (6.95,-1.4) and (3.31,-0.3) .. (0,0) .. controls (3.31,0.3) and (6.95,1.4) .. (10.93,3.29)   ;
    \draw    (166.67,58.17) -- (202.33,58.17) ;
    \draw [shift={(164.67,58.17)}, rotate = 0] [color={rgb, 255:red, 0; green, 0; blue, 0 }  ][line width=0.75]    (10.93,-3.29) .. controls (6.95,-1.4) and (3.31,-0.3) .. (0,0) .. controls (3.31,0.3) and (6.95,1.4) .. (10.93,3.29)   ;
    \draw    (206.87,68.7) -- (160.33,112) ;
    \draw [shift={(208.33,67.33)}, rotate = 137.06] [color={rgb, 255:red, 0; green, 0; blue, 0 }  ][line width=0.75]    (10.93,-3.29) .. controls (6.95,-1.4) and (3.31,-0.3) .. (0,0) .. controls (3.31,0.3) and (6.95,1.4) .. (10.93,3.29)   ;
    \draw    (151,103.67) -- (151,69.67) ;
    \draw [shift={(151,105.67)}, rotate = 270] [color={rgb, 255:red, 0; green, 0; blue, 0 }  ][line width=0.75]    (10.93,-3.29) .. controls (6.95,-1.4) and (3.31,-0.3) .. (0,0) .. controls (3.31,0.3) and (6.95,1.4) .. (10.93,3.29)   ;
    \draw    (275.67,56.83) -- (311.33,56.83) ;
    \draw [shift={(273.67,56.83)}, rotate = 0] [color={rgb, 255:red, 0; green, 0; blue, 0 }  ][line width=0.75]    (10.93,-3.29) .. controls (6.95,-1.4) and (3.31,-0.3) .. (0,0) .. controls (3.31,0.3) and (6.95,1.4) .. (10.93,3.29)   ;
    \draw    (260,103) -- (260,69) ;
    \draw [shift={(260,105)}, rotate = 270] [color={rgb, 255:red, 0; green, 0; blue, 0 }  ][line width=0.75]    (10.93,-3.29) .. controls (6.95,-1.4) and (3.31,-0.3) .. (0,0) .. controls (3.31,0.3) and (6.95,1.4) .. (10.93,3.29)   ;
    \draw    (312,119) -- (273.33,119) ;
    \draw [shift={(314,119)}, rotate = 180] [color={rgb, 255:red, 0; green, 0; blue, 0 }  ][line width=0.75]    (10.93,-3.29) .. controls (6.95,-1.4) and (3.31,-0.3) .. (0,0) .. controls (3.31,0.3) and (6.95,1.4) .. (10.93,3.29)   ;
    \draw    (215.2,110.4) .. controls (208.5,101.1) and (212.88,79) .. (215.13,69.84) ;
    \draw [shift={(215.6,68)}, rotate = 104.74] [color={rgb, 255:red, 0; green, 0; blue, 0 }  ][line width=0.75]    (10.93,-3.29) .. controls (6.95,-1.4) and (3.31,-0.3) .. (0,0) .. controls (3.31,0.3) and (6.95,1.4) .. (10.93,3.29)   ;
    \draw    (222.8,68.8) .. controls (226.62,82.17) and (224.92,96.41) .. (222.37,109.03) ;
    \draw [shift={(222,110.8)}, rotate = 281.9] [color={rgb, 255:red, 0; green, 0; blue, 0 }  ][line width=0.75]    (10.93,-3.29) .. controls (6.95,-1.4) and (3.31,-0.3) .. (0,0) .. controls (3.31,0.3) and (6.95,1.4) .. (10.93,3.29)   ;
    \draw    (324,108.4) .. controls (317.3,99.1) and (321.68,77) .. (323.93,67.84) ;
    \draw [shift={(324.4,66)}, rotate = 104.74] [color={rgb, 255:red, 0; green, 0; blue, 0 }  ][line width=0.75]    (10.93,-3.29) .. controls (6.95,-1.4) and (3.31,-0.3) .. (0,0) .. controls (3.31,0.3) and (6.95,1.4) .. (10.93,3.29)   ;
    \draw    (331.6,66.8) .. controls (335.42,80.17) and (333.72,94.41) .. (331.17,107.03) ;
    \draw [shift={(330.8,108.8)}, rotate = 281.9] [color={rgb, 255:red, 0; green, 0; blue, 0 }  ][line width=0.75]    (10.93,-3.29) .. controls (6.95,-1.4) and (3.31,-0.3) .. (0,0) .. controls (3.31,0.3) and (6.95,1.4) .. (10.93,3.29)   ;
    
    \draw (37.17,51.4) node [anchor=north west][inner sep=0.75pt]  [font=\footnotesize]  {$1$};
    \draw (105.17,113.57) node [anchor=north west][inner sep=0.75pt]  [font=\footnotesize]  {$3$};
    \draw (38.17,114.57) node [anchor=north west][inner sep=0.75pt]  [font=\footnotesize]  {$4$};
    \draw (105.17,51.4) node [anchor=north west][inner sep=0.75pt]  [font=\footnotesize]  {$2$};
    \draw (67.27,139.97) node [anchor=north west][inner sep=0.75pt]  [font=\footnotesize]  {$G$};
    \draw (146.5,51.4) node [anchor=north west][inner sep=0.75pt]  [font=\footnotesize]  {$1$};
    \draw (214.5,113.57) node [anchor=north west][inner sep=0.75pt]  [font=\footnotesize]  {$3$};
    \draw (147.5,114.57) node [anchor=north west][inner sep=0.75pt]  [font=\footnotesize]  {$4$};
    \draw (214.5,51.4) node [anchor=north west][inner sep=0.75pt]  [font=\footnotesize]  {$2$};
    \draw (255.5,50.73) node [anchor=north west][inner sep=0.75pt]  [font=\footnotesize]  {$1$};
    \draw (323.5,112.9) node [anchor=north west][inner sep=0.75pt]  [font=\footnotesize]  {$3$};
    \draw (256.5,113.9) node [anchor=north west][inner sep=0.75pt]  [font=\footnotesize]  {$4$};
    \draw (323.5,50.73) node [anchor=north west][inner sep=0.75pt]  [font=\footnotesize]  {$2$};
    \draw (149.6,139.97) node [anchor=north west][inner sep=0.75pt]  [font=\footnotesize]  {$\pi \ =\ ( 1\ 2\ 4)$};
    \draw (254.93,139.97) node [anchor=north west][inner sep=0.75pt]  [font=\footnotesize]  {$\pi \ =\ ( 1\ 2\ 3\ 4)$};

    \end{tikzpicture}

    \caption{Permutations that factor according to vertex-disjoint cycles in $G$.}
\label{fig:non-cycle-disjoint}
\end{figure}

\end{exmp}

We  remark that if $\pi$ factors according to vertex-disjoint cycles in $G$, then the graphs $G$ and $\pi(G)$ have the same sets of $p$-adjacencies as defined in~\cite{DBLP:conf/uai/Richardson96}.

\noindent The theorem below characterizes all graphs distribution equivalent to a given directed graph $G$.  Here, we write $=_d$ for equality in distribution.

\begin{thm} \label{distrequivalence}
Two directed graphs $G = (V, E)$ and $G' = (V, E')$ are distribution equivalent if and only if $G'=\pi(G)$ for a permutation $\pi$ on $V$ that factors according to vertex-disjoint cycles in $G$. Moreover, given $\Lambda\in\mathbb R^E_*$ and a noise vector $\varepsilon$ with independent components for the model corresponding to $G$, the permutation $\pi$ uniquely determines a coefficient matrix $\Lambda'\in\mathbb R^{E'}_*$ and a noise vector $\eps'$ with independent components for the model given by $G'$ such that $X':=(I-\Lambda')^\rT\eps'=_d(I-\Lambda)^{-\rT}\eps=:X$.
\end{thm}
\begin{proof} We first prove the necessity of our condition. A candidate $(\Lambda',\eps')$ must satisfy $\eps'=_d(I-\Lambda')^\rT(I-\Lambda)^{-\rT}\eps$. Since $\eps'$ must have independent components and the entries of $\varepsilon$ are non-Gaussian, the Darmois-Skitovich Theorem~\citep{darmois1953,skitovich1953} implies that every row of $(I-\Lambda')^\rT(I-\Lambda)^{-\rT}$ has precisely one nonzero entry. In other words, there exists a permutation matrix $P$ and a dilation matrix $D$ so that
\begin{equation} \label{perm}
  (I-\Lambda)PD=I-\Lambda'.
\end{equation} 
We can write $I-\Lambda=(-\lambda_{ij})_{i,j\in V}$ with $\lambda_{ii}=-1$ for $i=1,\ldots,n$. Suppose that right multiplication by $P$ permutes the columns of $I-\Lambda$ via the permutation $\pi$, so that the diagonal entries of $(I-\Lambda)P$ are $-\lambda_{i\pi(i)}$. But $I-\Lambda'$ has diagonal entries equal to $1$, so we must have $\lambda_{i\pi(i)}\ne 0$. This implies that $\pi(i)=i$ or $i\to\pi(i)\in E$. In particular, $\pi$ factors according to vertex-disjoint cycles in $G$. Moreover, since $I-\Lambda'$ has diagonal entries equal to $1$, by \rf{perm}, the $i$-th diagonal entry of $D$ must equal $-1/\lambda_{i\pi(i)}$. Hence, if we write $I-\Lambda'=(-\lambda'_{ij})_{i,j\in V}$ with $\lambda'_{ii}=-1$ for $i=1,\ldots,n$ as before, then
$$ -\lambda'_{ij}=\begin{cases} \lambda_{i\pi(j)}/\lambda_{j\pi(j)}, & \text{if}\ i\ne \pi(j), \\ -1/\lambda_{j\pi(j)}, & \text{if}\ i=\pi(j).\end{cases} $$
Since $\lambda_{ij}\ne 0$ if and only if $i\to j\in E$, we have that $G'=\pi(G)$ as required. 

To prove the sufficiency of our condition, we can simply calculate that $(I-\Lambda)^{-\rT}\eps=_d(I-\Lambda')^{-\rT}\eps'$, where $\Lambda'$ is determined as above, and $\eps'=_d(I-\Lambda')^\rT(I-\Lambda)^{-\rT}\eps=DP^\rT\eps$.
\hfill$\square$ 
\end{proof}

In our Definition~\ref{def:equivalence} of equivalence, we can rearrange the  structural equations,  and thus transform the graph only if the $\Lambda$ coefficients we divide by in our procedure are nonzero, which is satisfied whenever $\Lambda\in\mathbb R^E_{*}$. However, if some of the $\lambda_{ij}$ are zero for an edge $i\to j\in E$, then the random variable does not necessarily lie in the transformed model. 

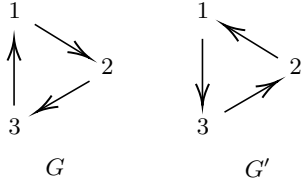
\begin{figure} 
    \centering
    \tikzset{every picture/.style={line width=0.6pt}} 
    
    \begin{tikzpicture}[x=0.75pt,y=0.75pt,yscale=-1,xscale=1]
    
    \draw    (100,77) -- (125.64,92.94) ;
    \draw [shift={(127.33,94)}, rotate = 211.88] [color={rgb, 255:red, 0; green, 0; blue, 0 }  ][line width=0.75]    (10.93,-3.29) .. controls (6.95,-1.4) and (3.31,-0.3) .. (0,0) .. controls (3.31,0.3) and (6.95,1.4) .. (10.93,3.29)   ;
    \draw    (127.33,105.67) -- (102.39,120.32) ;
    \draw [shift={(100.67,121.33)}, rotate = 329.57] [color={rgb, 255:red, 0; green, 0; blue, 0 }  ][line width=0.75]    (10.93,-3.29) .. controls (6.95,-1.4) and (3.31,-0.3) .. (0,0) .. controls (3.31,0.3) and (6.95,1.4) .. (10.93,3.29)   ;
    \draw    (89.67,118) -- (89.67,86.33) ;
    \draw [shift={(89.67,84.33)}, rotate = 90] [color={rgb, 255:red, 0; green, 0; blue, 0 }  ][line width=0.75]    (10.93,-3.29) .. controls (6.95,-1.4) and (3.31,-0.3) .. (0,0) .. controls (3.31,0.3) and (6.95,1.4) .. (10.93,3.29)   ;
    \draw    (195.7,78.06) -- (221.33,94) ;
    \draw [shift={(194,77)}, rotate = 31.88] [color={rgb, 255:red, 0; green, 0; blue, 0 }  ][line width=0.75]    (10.93,-3.29) .. controls (6.95,-1.4) and (3.31,-0.3) .. (0,0) .. controls (3.31,0.3) and (6.95,1.4) .. (10.93,3.29)   ;
    \draw    (219.61,106.68) -- (194.67,121.33) ;
    \draw [shift={(221.33,105.67)}, rotate = 149.57] [color={rgb, 255:red, 0; green, 0; blue, 0 }  ][line width=0.75]    (10.93,-3.29) .. controls (6.95,-1.4) and (3.31,-0.3) .. (0,0) .. controls (3.31,0.3) and (6.95,1.4) .. (10.93,3.29)   ;
    \draw    (183.67,116) -- (183.67,84.33) ;
    \draw [shift={(183.67,118)}, rotate = 270] [color={rgb, 255:red, 0; green, 0; blue, 0 }  ][line width=0.75]    (10.93,-3.29) .. controls (6.95,-1.4) and (3.31,-0.3) .. (0,0) .. controls (3.31,0.3) and (6.95,1.4) .. (10.93,3.29)   ;
    
    \draw (85.5,123.23) node [anchor=north west][inner sep=0.75pt]  [font=\footnotesize]  {$3$};
    \draw (85.5,64.57) node [anchor=north west][inner sep=0.75pt]  [font=\footnotesize]  {$1$};
    \draw (131.67,93.07) node [anchor=north west][inner sep=0.75pt]  [font=\footnotesize]  {$2$};
    \draw (104,143.23) node [anchor=north west][inner sep=0.75pt]  [font=\small]  {$G$};
    \draw (203.33,143.23) node [anchor=north west][inner sep=0.75pt]  [font=\small]  {$G'$};
    \draw (179.5,123.23) node [anchor=north west][inner sep=0.75pt]  [font=\footnotesize]  {$3$};
    \draw (179.5,64.57) node [anchor=north west][inner sep=0.75pt]  [font=\footnotesize]  {$1$};
    \draw (225.67,93.07) node [anchor=north west][inner sep=0.75pt]  [font=\footnotesize]  {$2$};

    \end{tikzpicture}

    \caption{The $3$-cycles of Example \ref{ex:zero_lambdas}}
    \label{fig:3cycles}
\end{figure}

\begin{exmp}\label{ex:zero_lambdas}
Consider the graphs  $G$ and $G'$ in Figure~\ref{fig:3cycles}. 
Let $\Lambda\in\mathbb R^E$ satisfy the sparsity pattern arising from $G$ but have additional zeros so that it no longer lies in $\mathbb R^E_*$:
$$\Lambda = \begin{pmatrix} 0 & 3 & 0\\ 0&0&0\\0&0&0
\end{pmatrix}.$$
Let $\Omega^{(2)}=\Omega^{(3)}=$ diag$(1,2,1)$.
Then, the second- and third-order moments $S$ and $T$ are
$$S = \begin{pmatrix}1&3&0\\3&11&0\\0&0&1
\end{pmatrix}, \quad T_{1..} = \begin{pmatrix}1&3&0\\3&9&0\\0&0&0
\end{pmatrix},$$
$$\quad T_{2..} = \begin{pmatrix}3&9&0\\9&29&0\\  0&0&0
\end{pmatrix}, \quad T_{3..} = \begin{pmatrix}0&0&0\\0&0&0\\0&0&1
\end{pmatrix},$$
and $(S, T)\in\mathcal M^{(2,3)}(G)$. We will show that $(S, T)\not\in\mathcal M^{(2,3)}(G')$.
Suppose that there exist
$$\Lambda' = \begin{pmatrix} 0 & 0 & \lambda'_{13}\\ \lambda'_{21}&0&0\\0&\lambda'_{32}&0
\end{pmatrix},$$
  $\Omega^{(2)'}=\text{diag}(\omega'_{21},\omega'_{22},\omega'_{23})$,
$\Omega^{(3)'}=\text{diag}(\omega'_{31},\omega'_{32},\omega'_{33})$

such that $S = (I-\Lambda')^{-T}\Omega^{(2)'}(I-\Lambda)^{-1}$ and $T = \Omega^{(3)'}\bullet (I-\Lambda')^{-1}\bullet (I-\Lambda')^{-1}\bullet (I-\Lambda')^{-1}$. Taking inverses,
$$S^{-1} = \begin{pmatrix}\frac{11}2&-\frac32&0\\-\frac32&\frac12&0\\0&0&1
\end{pmatrix} =$$
$$=\begin{pmatrix}\frac1{\omega'_{21}}+\frac{(\lambda'_{13})^2}{\omega'_{23}} & -\frac{\lambda'_{21}}{\omega'_{  21}} & -\frac{\lambda'_{13}}{\omega'_{23}}\\
-\frac{\lambda'_{21}}{\omega'_{  21}} & \frac{(\lambda'_{21})^2}{\omega'_{21}} + \frac1{\omega'_{22}}& -\frac{\lambda'_{32}}{\omega'_{22}}\\
-\frac{\lambda'_{13}}{\omega'_{23}} &-\frac{\lambda'_{32}}{\omega'_{22}} & \frac{(\lambda'_{32})^2}{\omega'_{22}} + \frac1{\omega'_{23}}
\end{pmatrix}.$$
This leads to the solution
$$\lambda'_{21}=\frac3{11}, \lambda'_{13}  = 0, \lambda'_{32} = 0, \omega'_{    21} = \frac2{11}, \omega'_{22} = 11, \omega'_{23} = 1.$$
Equating $T = \Omega^{(3)'} \bullet (I-\Lambda')^{-1}\bullet (I-\Lambda')^{-1}\bullet (I-\Lambda')^{-1}$, we get
$$\begin{pmatrix}1&3&0\\3&9&0\\0&0&0
\end{pmatrix} = \begin{pmatrix}\omega'_{31} +  \frac{27\omega'_{32}}{1331}& \frac{9\omega'_{32}}{121} & 0\\
\frac{9\omega'_{32}}{121} & \frac{3\omega'_{32}}{11} & 0\\
0&0&0
\end{pmatrix},\,\,
 \begin{pmatrix}3&9&0\\9&29&0\\  0&0&0
\end{pmatrix}$$
$$=\begin{pmatrix}\frac{9\omega'_{32}}{121} & \frac{3\omega'_{32}}{11} & 0\\
\frac{3\omega'_{32}}{11} & \omega'_{32} & 0\\
0&0&0
\end{pmatrix},\,\,
\begin{pmatrix}0&0&0\\0&0&0\\0&0&1
\end{pmatrix} = \begin{pmatrix}0&0&0\\
0&0&0\\
0&0&\omega'_{33}
\end{pmatrix}.$$
This set of linear equations in $\omega'_{31},  \omega'_{32}, \omega'_{33}$ does not have a solution. For example, we have $3 = \frac{9\omega'_{32}}{121}, 9 = \frac{3\omega'_{32}}{11}$, which has no solution. Therefore, we cannot transform the equations so that $(S, T)\in\mathcal M^{(2,3)}(G')$. 
\end{exmp}

\section{Discovery algorithm for cycle-disjoint graphs}\label{sec:algorithm}

This section presents the concepts and results essential for our
discovery algorithm.
We begin by formally defining the considered graphs.

\begin{defn}\label{def:cycledisjoint} A directed graph $G$ is {\em
    cycle-disjoint} if each one of its nodes belongs to at most one cycle in $G$.  
\end{defn}

For example, the graphs in Figure \ref{fig:EquivalentGraphs} are cycle-disjoint, while those in Figure \ref{fig:non-cycle-disjoint} are not.  Note that a cycle-disjoint graph $G$ has all its strong components equal to directed cycles (without any additional edges in the components).

%
{The general Theorem \ref{distrequivalence} allows us to particularly easily characterize all graphs that are distribution equivalent to a graph with disjoint cycles. This can be done by reversing the orientation of any of the disjoint cycles and adjusting the incoming edges.}

Our discovery algorithm has two distinct steps:

\textbf{Step 1.} Discover the strong components of the graph $G$, their topological ordering, and the edges and edge weights within each component.

\textbf{Step 2.} Discover the edges and edge weights between different components.

Below, we prove that these steps can be implemented via algebraic constraints among the second and third moments.

\begin{rmk}
\label{rmk:consistency}\rm
Causal discovery algorithms that process algebraic relations among
moments can be practically implemented using sample moments calculated
from data.  Comparing sample quantities to thresholds, one may decide
algorithmically whether an algebraic constraint of interest holds.
A suitable choice of a threshold renders such algorithms consistent
because sample moments consistently estimate true moments when the
available sample size tends to infinity.
\end{rmk}

In the rest of this section we assume that $G=(V,E)$ is a cycle-disjoint graph and that $(S,T)\in \mathcal{M}^{2,3}(G)$ are the second- and third-order moments of a random vector following the LSEM given by $G$.  We will use further notation as follows. In graph $G=(V,E)$, the set of
ancestors of a subset $C\subseteq V$, denoted by $\an(C)$, is the set
of nodes $u\in V$ with a directed path (of length at least $0$) from
$u$ to some $v\in C$. We say that $C$ is \emph{ancestral} if
$\an(C)= C$. Finally, a node $v\in V$ is a \emph{root (node)} if $\{v\}$
is ancestral and a cycle $C$ is a \emph{root cycle} if it is
ancestral.

All complete proofs of the results that follow are provided in the supplementary material~\ref{sec:proofs_algorithm}.

\subsection{Discovering the strong components}

For two vertices $u,v\in V$, we define the determinants
  \begin{align*}\label{eq:determinants}
      d_{uv}^{2\times 2}&=\det\begin{pmatrix}s_{uu}&s_{uv}\\ t_{uuu}&t_{uuv}\end{pmatrix}, \\
      d_{uv}^{3\times 3}&=\det\begin{pmatrix}s_{uu}&s_{uv}&s_{vv}\\t_{uuu}&t_{uuv}&t_{uvv}\\t_{uuv}&t_{uvv}&t_{vvv}\end{pmatrix}.
  \end{align*}
Note that 
$d_{uv}^{3\times 3} = d_{vu}^{3\times 3}$, 
but $d_{uv}^{2\times 2} \neq d_{vu}^{2\times 2}$ in general.

\begin{thm}\label{d2d3}The determinants $d_{uv}^{2\times 2}$ and $d_{uv}^{3\times 3}$
satisfy the following relationships.
  \begin{enumerate}[label=(\alph*)]
  \item $d_{uv}^{2\times 2}=0$ identically if and only if $u$ and $v$ have no common ancestors and $v$ is not an ancestor of $u$;
  \item $d_{uv}^{3\times 3}=0$ identically if and only if there exists no {simple 2-trek $(P_u,P_v)\in\cT(u,v)$ with non-empty sides.}  
  \end{enumerate}
  By ``identically'' we mean for any choice of $\Lambda\in\mathbb R^{E}$ and distribution on $\eps$.
\end{thm}
We remark that conditions $(a)$ and $(b)$ can be restated using
trek-separation~\citep{sullivant2010trek}; cf.~Lemma~\ref{lem:trek}.

The previous theorem allows us to identify the root nodes of the graph (Corollary~\ref{cor:d2d3}) as well as a collection of cycles containing all root cycles (Corollary~\ref{cor:rootCycles}).  

\begin{cor} \label{cor:d2d3}
  The determinant $d_{ru}^{2\times 2}=0$ identically {for all $u\in V$} if and only if $r$ is a root of $G$.
\end{cor}
 
 \begin{cor}\label{cor:rootCycles}
  Set $\cC$ to be the collection of all maximal $C\subseteq V$ such that for all $u,v\in C$, we have $d_{uv}^{3\times 3}=0$ identically but $d_{uv}^{2\times 2}\ne 0$. Then, $\cC$ contains all root {cycles}.
\end{cor}

To precisely identify the root 
 cycles in the set of candidate root cycles $\mathcal C$,  we use
 Proposition~\ref{prop:regression} and Lemma~\ref{lem:rootCycle}. To state them, we first need a definition.
 
 \begin{defn}
     We denote by $R_{A,B}$ the set of regression coefficients obtained when regressing the variables in $A$ on the variables in $B$.
 \end{defn}
 
 Proposition~\ref{prop:regression} shows  that if we regress all remaining variables on an ancestral set $C$ (e.g., a root cycle), the regression  residuals  will  follow the LSEM corresponding to the graph obtained by removing the vertices $C$ and all their adjacent edges from the graph $G$, denoted by $G[V\setminus C]$.
 
\begin{prop}\label{prop:regression}
    Let $X=(I-\Lambda)^{-T}\varepsilon$ be a random vector whose distribution is in the LSEM given by the graph $G$. 
  Let $C\subseteq V$ be an ancestral set in $G$, and define the vector of residuals
  \[
  X_{V\setminus C.C} := X_{V\setminus C}-R_{V\setminus C,C}X_C.
  \]   
 Then the joint distribution of $X_{V\setminus C.C}$ is in the LSEM given by the induced subgraph $G[V\setminus C]$, with
  \[
  X_{V\setminus C.C} = (I-\Lambda_{V\setminus C,V\setminus C})^{-T}\varepsilon_{V\setminus C}.
  \]
\end{prop}

Now, to detect the root cycles within 
 $\mathcal C$, we take two  candidate root cycles $C,D\in\mathcal C$,
 and regress $D$ on $C$. If $C$ is indeed a root cycle, then $X_{D.C}$
 should be independent of $X_C$. We check such independence using the
 third-order moments (assuming existence of skew).

\begin{lem}\label{lem:rootCycle}
For $\cC$ from Corollary \ref{cor:rootCycles}, and $C,D\in\cC$,
define the residuals $X_{D.C} = X_D - R_{D,C}X_C$.
Then $C$ is a root cycle if and only if the matrix $(\EE[X_{c}^2(X_{D.C})_d])_{c\in C,d\in D}$ is identically zero  for all $D\in\mathcal C\setminus C$.
\end{lem}

Therefore, we can recursively identify root nodes and root cycles, regress them away, and continue identifying more root nodes and root cycles until  we have gone through all nodes. This process finds not only the strong components, but also a topological ordering among them.

\subsection{Discovering all edges and edge weights in cycles} 
  When identifying root cycles as above, we can also identify all edges and edge  weights within each such  cycle. 
First, to identify the skeleton of a root cycle, we  use the following result, which follows from \citet[Eqn.~(6.1)]{drtonFKP:2019}.

\begin{lem}\label{lem:condIndep}
Let $C$ be a root cycle.
Two vertices $u$ and $v$ of $C$ are not adjacent if and only if the $(i,j)$-entry of the inverse of $S_{C,C}$ is identically zero.  Equivalently,
$i$ and $j$ are not adjacent if and only if $\det S_{C\setminus j, C\setminus i} = 0$ identically.
\end{lem}


Once we have found the skeleton of our root cycle, we choose any one of the two possible orientations. According to Theorem~\ref{distrequivalence}, each of those is possible in a representative of the equivalence class. We then proceed to find the edge weights on the cycle. The following lemma applies in the case of cycles of length 3 or more.

\begin{lem}\label{lem:rankAs}
  Let $u,v,w$ be three consecutive nodes in a root cycle $C$ and $\lambda_{uv}$ the corresponding edge weight in edge $u\to v$. Then, the matrices 
  \begin{align*}\label{eq:rankAs}
    A_{uv}^{(2)} &= \left(\begin{array}{ccc}
      1 & \lambda_{uv} & \lambda_{uv}^2 \\
      s_{uu} & s_{uv} & s_{vv} \\
      t_{uuu} & t_{uuv} & t_{uvv} \\
      t_{uuv} & t_{uvv} & t_{vvv} 
    \end{array}\right) \text{ and  }\\
    A_{uvw}^{(3)} &= \left(\begin{array}{ccccc}
      1 & \lambda_{uv}  & \lambda_{uv}^2 & \lambda_{uv}^2 & \lambda_{uv}^3 \\
      s_{uu} & s_{uv}   & s_{vv}  & s_{uw}  & s_{vw} \\
      t_{uuu} & t_{uuv} & t_{uvv} & t_{uuw}  & t_{uvw} \\
      t_{uuv} & t_{uvv} & t_{vvv} & t_{uvw}  & t_{vvw}
    \end{array}\right)
  \end{align*}
  have rank at most $2$ and $3$, respectively.  For generic choices of $\Lambda$, $\Omega^{(2)}$, and $\Omega^{(3)}$, the respective ranks equal~2~and~3.
\end{lem}

This  allows us to find simple linear (for cycles of length 3 or more) or  quadratic (for cycles of length 2) equations in the edge weights $\lambda_{uv}$ along the cycle.
\begin{thm}\label{thm:lambda01}
  Let $C$ be a {root} cycle in $G$. Then 
  \begin{enumerate}[label=(\roman*)]
  \item If the cycle length is at least 3, and $u,v,w\in C$ are three consecutive nodes on $C$, then there is a linear equation in $\lambda_{uv}$, 
  \begin{equation}\label{eq:linEqlambda} p(s, t)\lambda_{uv} = q(s, t),\end{equation} where 
  \begin{align*}
  p(s,t) &= 
  s_{uu}(t_{uvw}^2-t_{uuw}t_{vvw}) + s_{uw}(t_{uuu}t_{vvw}\\&-t_{uuv}t_{uvw}) + s_{vw}(t_{uuv}t_{uuw}-t_{uuu}t_{uvw}),\\
    q(s,t) &= - s_{uv}(t_{uvw}^2-t_{uuw}t_{vvw}) + s_{uw}(t_{uvv}t_{uvw}\\&-t_{uuv}t_{vvw}) + s_{vw}(t_{uuv}t_{uvw}-t_{uuw}t_{uvv}).
  \end{align*}
  \item If the cycle length is 2 with $u,v\in C$, then there are two solutions for the edge coefficient $\lambda_{uv}$ given by the  quadratic equation 
  \begin{align}
  \nonumber
  &(s_{uu}t_{uuv}-s_{uv}t_{uuu})\lambda_{uv}^2+(s_{vv}t_{uuu}\\&-s_{uu}t_{uvv})\lambda_{uv}+s_{uv}t_{uvv}-s_{vv}t_{uuv} = 0.
  \label{eq:linEqlambda2}
  \end{align} 
\end{enumerate}
\end{thm}

Note that in the case of cycles of length at least 3, 
 these equations do not depend on the length of the cycle since they only use the moments of three of the vertices.

\subsection{Discovering edges and edge weights between strong components}
Once we have found all strong components, the edges and edge weights within each components, and a topological ordering among the strong components, we proceed to find the edges and edge weights between different strong components. This is done via a simple regression, which uses the topological ordering, and simply regresses each cycle on all cycles before it in topological ordering. The edge weights are then equal to the regression coefficients adjusted by the cycle effect, as shown in the lemma below.

\begin{lem}\label{thm:R_DC}
  Let $D$ be a cycle in $G$ and $C\subset V$ the vertices of all cycles that appear before $D$ in topological ordering. {Then the parameters $\lambda_{cd}$ with $c\in C$ and $d\in D$ can be calculated using the regression coefficient $R_{D,C}$ and the edge weights $\Lambda_{D,D}$ within the cycle $D$ via $$ \Lambda_{C, D} = R_{DC}^T(I - \Lambda_{D, D}).$$}
\end{lem}
\begin{proof} 
The result follows directly since $R_{D,C} = (I - \Lambda_{D, D})^{-T}(\Lambda_{C, D})^T$; see Lemma \ref{lem:refCoeffs} for more details.\hfill$\square$
\end{proof}

\subsection{Algorithm}
Combining the results from the previous subsections,  we present our full algorithm in Algorithm~\ref{discAlg}.
The implementation details can be found in the supplementary material~\ref{sec:algodetails} {and the code is available in the repository: \href{https://github.com/ysamwang/disjointCycles}{github.com/ysamwang/disjointCycles}.} 

Remark~\ref{rmk:consistency} discusses consistency of the proposed procedure when a given threshold allows one to determine whether quantities are zero or nonzero. In practice, this threshold is not known a priori, so our procedure instead uses a hypothesis test to certify whether each quantity of interest is zero or nonzero. Thus, $\alpha$, the level for each hypothesis test is a tuning parameter we must specify.
In addition, because we perform a large number of tests, we find it empirically advantageous to use either the Bonferroni-Holm~\citep{holm1979multiple} or Benjamini-Hochberg~\citep{benjamini1995fdr} procedures to control false positives at various stages of the algorithm. Additional details are given in the supplement. 

\begin{algorithm}[h]
 \caption{\label{discAlg}Causal Graph Discovery}\label{alg:disc}
 \begin{algorithmic}[1]
  \Require Random vector $X=(X_i)_{i}$ on $n$ components with moments $S=\EE[X_iX_j]$ and $T=\EE[X_iX_jX_k]$
  \Ensure Causal DG $G=(V,E)$ and edge weights $\Lambda$.
  \hrule
  \begin{flushleft}
    \it \textbf{Part 1.} Learning strong components, their order, and edges and edge weights within components
  \end{flushleft}
  \State Until the graph is nonempty, compute all $d_{uv}^{2\times 2}$ and $d_{uv}^{3\times 3}$.
  \State Set $\cC$ to be the collection of all vertices $r\in V$ such that such that $d_{ur}^{2\times 2}= 0$ for all $u\in V$, $\cC$ is the set of all root nodes (Corollary \ref{cor:d2d3}). Record $\mathcal C$.
  \State If $\cC\not=\emptyset$, regress $V\setminus\cC$ on $\cC$ to obtain samples from graph without $\mathcal C$ (Proposition~\ref{prop:regression}); go back to Step 1.
  \State Else, set $\cC$ to be the collection of maximal $C\subseteq V$ such that $d_{uv}^{3\times 3}=0$ and $d_{uv}^{2\times 2}\ne 0$ for all $u,v\in C$
  \State Prune $\cC$ to obtain $\cC'$, all root cycles (
  Lemma~\ref{lem:rootCycle}). Record $\mathcal C'$.
  \State Compute edges and edge weights along each cycle in $\mathcal C'$ (Lemma~\ref{lem:condIndep} and Theorem~\ref{thm:lambda01}).
  \State Regress $V\setminus\cC'$ on $\cC'$ to obtain samples from graph without $\mathcal C'$ (Proposition~\ref{prop:regression}); go back to Step 1.
  \hrule
  \begin{flushleft}
      \it \textbf{Part 2.} Learning edges and weights between cycles
  \end{flushleft}
  \State Working backwards in topological ordering, regress each cycle on all possible parents to find edges between cycles (Lemma~\ref{thm:R_DC}).
  
 \end{algorithmic}
\end{algorithm} 
{Based on our previous results, we obtain the following.}

\begin{cor} {Let $G$ be a cycle-disjoint graph, and let $(S,T)\in \mathcal{M}^{2,3}(G)$ be the second- and third-order moments of a random vector following the LSEM based on $G$. Assume the coefficients $\Lambda$ and the moments $\Omega^{(2)}$, $\Omega^{(3)}$ of the noise vector $\varepsilon$ are generic. Then Algorithm~\ref{discAlg} recovers a graph $G'$ that is distribution equivalent to $G$, as well as appropriate coefficients $\Lambda'$ for $G'$. Applying Theorem~\ref{distrequivalence}, all graphs distribution equivalent to $G$ may be constructed.}
\end{cor}

{Note that after having recovered $G$ and $\Lambda$, we are also able to recover $\Omega^{(2)}$ and $\Omega^{(3)}$.}

{The genericity assumption needed when applying our algorithm to a particular distribution is that the distribution behaves like a typical member of the model with regards to the tested constraints. Given the polynomial nature of the constraints, only a measure zero set of parameters lead to distributions that violate such genericity.
}

\section{Numerical experiments}\label{sec:numerical_experiments}

\paragraph{Simulation setup} We now consider the empirical performance of the proposed procedure. For each replication we construct a random graph with $|V| = p$. We create $p/3$ disjoint cycles $C_1, C_2, \ldots, C_{p/3}$ in a total ordering where $\vert C_k \vert = 3$ for all $k$. We first include an edge between 1 node in cycle $C_k$ and 1 node in cycle $C_{k+1}$ so that $\an(C_k) = \cup_{j <k} C_j$. Then for all for $j < k$, and each $u \in C_{j}$ and $v \in C_k$, we add the edge $u \rightarrow v$ with probability $1/2$. The linear coefficients corresponding to each edge in the graph are drawn uniformly from $-(.8, .5)\cup (.5, .8)$. For the errors, we consider two distributions: a mixture of normals
with weight $.9$ on $N(-2, .1^2)$ and weight $.1$ on $N(2, .1^2)$ and
a Gamma(1,1) distribution. We center the distributions, and we scale each $\varepsilon_v$ so that the standard deviation is uniformly in $(.8, 1)$. Across the various settings, we let $n = 10000, 25000, 50000, 100000$.

\paragraph{ICA based procedure} We first show that the ICA based procedure of~\citet{DBLP:conf/uai/LacerdaSRH08} suffers computationally when $p$ is not small. Specifically, when the demixing matrix is estimated imprecisely, the procedure requires a search over an intractable number of permutations. To examine this concretely, we let $p= 18, 21, 24, 27$ and record the proportion of times the ICA based procedure returns an estimate within 24 hours. With gamma errors and $n = 100000$, the ICA based procedure---as implemented in the rpy-tetrad package~\citep{ramsey2023py}---returns an estimate in $14$ out of $25$ replicates when $p = 18$ and $10$ out of $25$ replicates when $p = 21$. The performance is worse when $n$ is smaller because the demixing matrix is estimated less precisely. When $p = 24, 27$ the procedure fails to return an estimate in all $25$ replicates; additional details are given in the appendix.

\paragraph{Proposed procedure} To demonstrate that our proposed procedure can handle larger problems, we apply it to problems where $|V| = 30, 45, 60$.  In the appendix, we also show results when the graph consists of cycles of size $5$. We assess the performance of our proposed procedure through two measures. 
First, we consider the proportion of times that Part 1 of Algorithm~\ref{discAlg} correctly identifies the cycles and their ordering. This ordering is invariant across graphs in the same equivalence class. We also count the number of pairs $u, v \in V \times V$ which are decided correctly; i.e., $u \rightarrow v$ vs.~$v \rightarrow u$ vs.~no edge. There may be multiple graphs in both the equivalence class of the ground truth and the equivalence class of the estimated graph. However, for the ground truth, there is a unique graph for which all cycles are stable; i.e, the product of the edges in the graph are less than $1$ in absolute value. By construction, this is the graph we used to generate the data. From the equivalence class of the estimated graph, we pick a unique graph in the following way. 
To orient the skeleton, we estimate the edges for both orientations
and select the orientation with the smaller pathweight. When the true
and estimated graphs are in the same equivalence class, this should
(up to sampling error when estimating the edgeweights) yield the same
unique graph. However, when the true and estimated graphs are in
different equivalence classes, this may not be the pair of graphs
which maximize the number of correct pairs. These measures are shown
in Figure~\ref{fig:perf}. For each type of graph and error distribution, we draw 100 replications; for
each setting, we show the combination of $\alpha$ and multiple testing correction which yields the best average performance.

The procedure performs better when the errors are mixtures of normals, which has lighter tails, as opposed to gamma. In addition, we see that as $p$ increases performance generally decreases. Nonetheless, even when $p = 60$, nearly 80\% of edges are correctly oriented at the largest sample size when the errors are mixtures of normals. In Figure~\ref{fig:time}, we also show the computational time in seconds. We see that the computation time increases with both $n$ and $p$. However, even at the largest sample size and $p = 60$, the average time is still under 1.25 hrs per replication.

\begin{figure}[t]
\centering
\includegraphics[scale = .55]{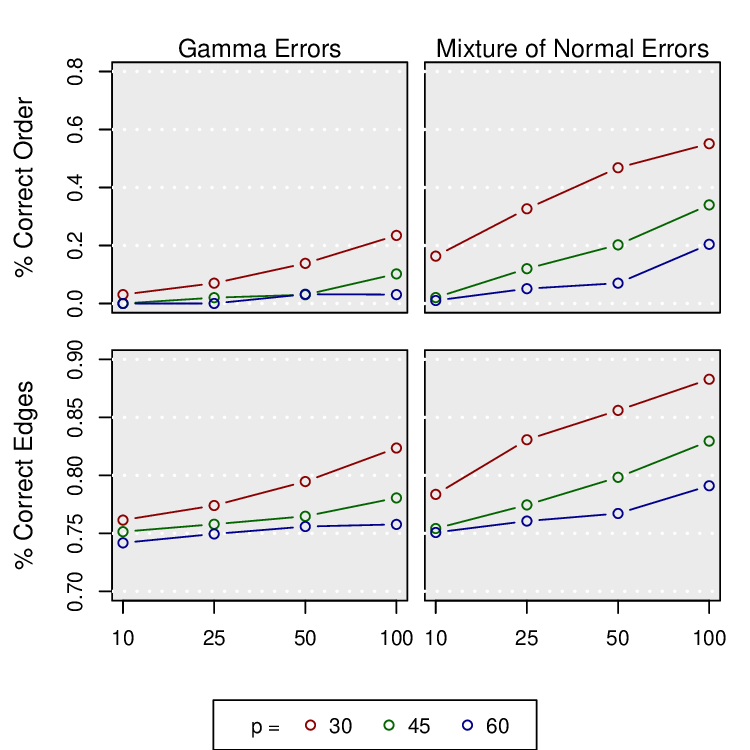}
\caption{\label{fig:perf}Horizontal axis shows sample size in thousands. Top: Proportion of times the correct ordering is recovered. Bottom: Proportion of edges that are correctly identified. Left: Errors are drawn from Gamma. Right: Errors are drawn from mixture of normals.}
\end{figure}

\begin{figure}[h]
\centering
\includegraphics[scale = .4]{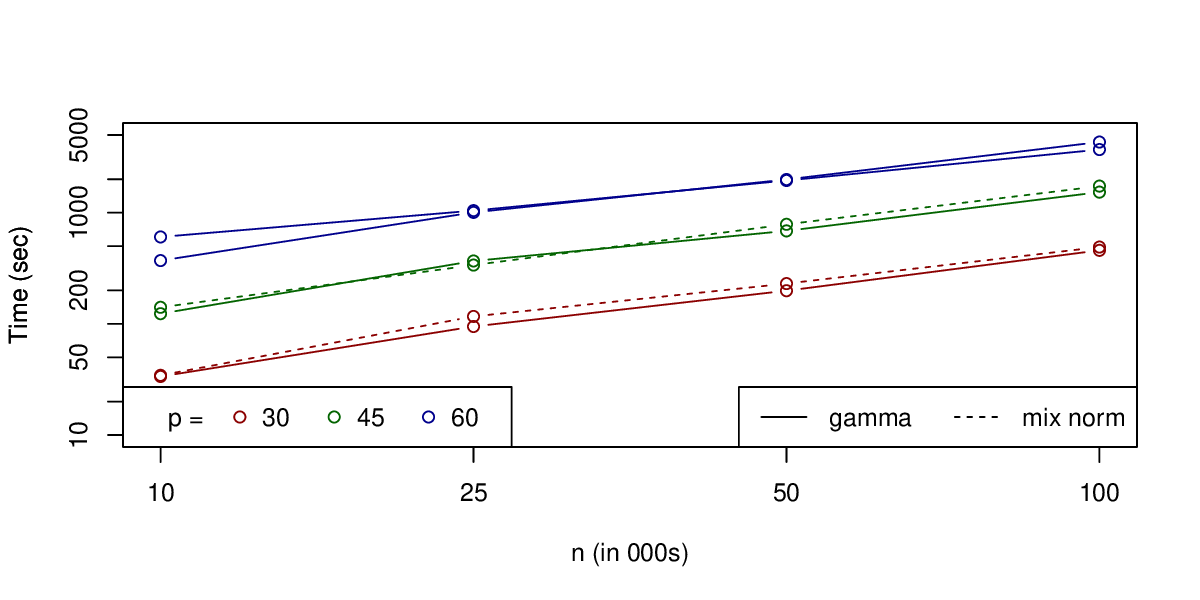}
\caption{\label{fig:time}Average time (in seconds) required per replication.}
\end{figure}

\section{Conclusion}
\label{sec:conclusion}

This paper studied causal discovery for linear non-Gaussian models that feature feedback loops.  Restricting to the setting of cycle-disjoint graphs, we derived algebraic constraints on moments that reveal the strong components of the graph, which then facilitates consistent causal discovery.

As shown in our experiments,  moment-based methods are state-of-the-art for higher-dimensional problems, in which more established ICA-based methods become prohibitively expensive in terms of computation time and memory.  From a theoretical perspective, moment-based algorithms also have the appeal of straightforward consistency guarantees.  

Our work proves that cycle-disjoint graphs can be learned computationally efficiently.  More algebraic insights are needed in order to design moment-based algorithms that could consistently learn non-cycle-disjoint cyclic graphs.



\begin{acknowledgements} 
    This project has received funding from the European Research Council (ERC) under the European Union’s Horizon 2020 research and innovation programme (grant agreement No 883818).
    %
\end{acknowledgements}

\bibliography{biblio}

\newpage

\onecolumn

\title{Causal Discovery for Linear Non-Gaussian Causal Models with Disjoint Cycles\\(Supplementary Material)}
\maketitle


\appendix

\section{Supplement to Section~\ref{sec:algorithm}}\label{sec:proofs_algorithm}

\subsection{The trek rule and trek separation}\label{sec:trekrule}
{As introduced in Section \ref{sec:preliminaries}, we use the trek rule to express the entries of $S$ and $T$ in terms of the entries of $\Lambda$, $\Omega^{(2)}$, and  $\Omega^{(3)}$:}
\begin{equation*}
        s_{ij}=\sum_{(P_i,P_j)\in\cT(i,j)}\lambda^{P_i}\lambda^{P_j}\omega^{(2)}_{\Top(P_i,P_j)}, \quad\text{and}\quad
        t_{ijk}=\sum_{(P_i,P_j,P_k)\in\cT(i,j,k)}\lambda^{P_i}\lambda^{P_j}\lambda^{P_k}\omega^{(3)}_{\Top(P_i,P_j,P_k)}, \nonumber
    \end{equation*}
    {where $\lambda^P=\lambda_{u_1u_2}\lambda_{u_2u_3}\ldots\lambda_{u_{r-1}u_r}$ for a path $P=(u_1,\ldots,u_r)$.}
We will use $\cP(i,j)$ to denote the set of $i\to j$-paths in $G$, and $\cP^\ast(i,j)\subseteq\cP(i,j)$ to denote those paths that have no repeating vertices. Moreover, the following notation \color{black}
$$ \lambda^{i\to j}=\sum_{P\in\cP(i,j)}\lambda^P \quad \text{and} \quad \lambda_\ast^{i\to j}=\sum_{P\in\cP^\ast(i,j)}\lambda^P $$
will appear in various factorizations of the $s_{ij}$ and $t_{ijk}$.

We further recall the definition of t-separation from \cite{sullivant2010trek}.

\begin{defn}
    Let $G=(V,E)$ be a directed graph and $A$, $B$, $C$, and $D$ subsets of $V$. A pair $(C,D)$  \emph{trek-separates} (or \emph{$t$-separates}) A from B if for every trek $(P_1,P_2)\in\cT(a,b)$ between any $a\in A$ and $b\in B$, either $P_1$ contains a vertex in $C$ or $P_2$ contains a vertex in $D$.
\end{defn}

\subsection{Trek-separation description of Theorem~\ref{d2d3}}

Conditions (a) and (b) of Theorem~\ref{d2d3} can be expressed using t-separation as follows.  
\begin{lem}\label{lem:trek} Let $G=(V,E)$ be a cycle-disjoint graph. Two vertices $u,v\in V$ have no common ancestor and $v$ is not an ancestor of $u$ if and only if $(\emptyset,u)$ $t$-separates $u$ from $v$. Similarly, there exists no {simple} trek $(P_u,P_v)\in\cT(u,v)$ with {non-empty sides} 
if and only if $(v,u)$ $t$-separates $u$ from $v$.
\end{lem}
\begin{proof} The nodes $u$ and $v$ share no common ancestor and $v$ is not an ancestor of $u$ if and only if there exists no {simple} trek $(P_u,P_v)\in\cT(u,v)$ with {non-empty paths} $P_u$ and $P_v$ 
and there exists no $v\to u$-path. This is equivalent to $(\emptyset,u)$ $t$-separating $u$ from $v$

  Similarly, there exists no {simple} trek $(P_u,P_v)\in\cT(u,v)$ with {non-empty} $P_u$ and $P_v$ 
  if and only if the only possible treks between $u$ and $v$ are either a directed path from $u$ to $v$ or a directed path from $v$ to $u$. This is equivalent to $(v,u)$ $t$-separating $u$ from $v$.
\end{proof}

\subsection{Proof of Theorem~\ref{d2d3}}
Before proving Theorem \ref{d2d3} we need the following lemma.

\begin{lem}\label{lem:d3cycle}
 $G$ is a cycle-disjoint graph, and $(S, T)\in\mathcal M^{(2,3)}(G)$. Let $C$ be a cycle in $G$ and $u,v$ two nodes in $C$ such that $(v,u)$ $t$-separates $u$ from $v$. Then, 
$ d_{uv}^{3\times 3}=\det\begin{pmatrix}s_{uu}&s_{uv}&s_{vv}\\t_{uuu}&t_{uuv}&t_{uvv}\\t_{uuv}&t_{uvv}&t_{vvv}\end{pmatrix}=0$.
\end{lem}
\begin{proof}
    By the trek rule \eqref{eq:trekrule}, the entries of the second and third-order moments $S$ and $T$, can be written as
    \begin{align}
        &s_{ij}=a_0^{(ij)} \omega_0^{(2)}+a_1^{(ij)} \omega_1^{(2)} + \cdots + a_n^{(ij)} \omega_n^{(2)}, \label{eq:sij} \\ 
        &t_{ijk}=b_0^{(ijk)} \omega_0^{(3)}+b_1^{(ijk)} \omega_1^{(3)} + \cdots + b_n^{(ijk)} \omega_n^{(3)} \nonumber
    \end{align}
    where 
    \begin{equation} \label{eq:a_ell}
        a_\ell^{(ij)} =\sum_{\substack{ (P_1,P_2)\in\cT(i,j) \\ \Top(P_1,P_2)=\ell}}\lambda^{P_1}\lambda^{P_2} =\lambda^{\ell\to i}\lambda^{\ell\to j} \quad \text{and} \quad b_\ell^{(ijk)} =\sum_{\substack{(P_1,P_2,P_3)\in\cT(i,j,k) \\ \Top(P_1,P_2,P_3)=\ell}}\lambda^{P_1}\lambda^{P_2}\lambda^{P_3} = \lambda^{\ell\to i}\lambda^{\ell\to j}\lambda^{\ell\to k}.
    \end{equation}

    We define the following sets of nodes for any two vertices $u,v\in C$
    \begin{align}\label{eq:AuAv}
    A_u &:= \left(\an(u)\setminus \an(v)\right) \cup \{w\in \an(u) \cap \an(v)\ |\ u\in P \text{ for all } P\in\cP^\ast(w,v) \} \text{ and}\\
    A_v &:= \left(\an(v)\setminus \an(u)\right) \cup \{w\in \an(u) \cap \an(v)\ |\ v\in P \text{ for all } P\in\cP^\ast(w,u) \}. \nonumber
    \end{align}
    Since $(v,u)$ $t$-separates $u$ from $v$, 
    there exists no {simple} trek $(P_u,P_v)\in\cT(u,v)$ with 
    {non-empty} paths $P_u$ and $P_v$. Consequently, $A_u\cup A_v = \an(u) \cup \an(v)$ and $A_u\cap A_v= \emptyset$.
    Therefore,
    \begin{equation*}
    \lambda^{\ell\to u} = 
        \begin{cases}
            \lambda^{\ell\to u} & \text{if } \ell\in A_u\\
            \lambda^{\ell\to v}\lambda^{v\to u}_* & \text{if } \ell\in A_v
        \end{cases} \text{ and }
    \lambda^{\ell\to v} = 
        \begin{cases}
            \lambda^{\ell\to u}\lambda^{u\to v}_* & \text{if } \ell\in A_u\\
            \lambda^{\ell\to v} & \text{if } \ell\in A_v
        \end{cases}
    \end{equation*}
    and then,
    \begin{align*}
        s_{uu} &= \sum_{\ell\in A_u} a_{\ell}^{(uu)}\omega_\ell^{(2)} + \sum_{\ell\in A_v} a_{\ell}^{(uu)}\omega_\ell^{(2)} = \sum_{\ell\in A_u} \left(\lambda^{\ell\to u} \right)^2\omega_\ell^{(2)} + \sum_{\ell\in A_v} \left(\lambda^{\ell\to u} \right)^2\omega_\ell^{(2)} = \\
        &= \sum_{\ell\in A_u} \left(\lambda^{\ell\to u} \right)^2\omega_\ell^{(2)} + \sum_{\ell\in A_v} \left(\lambda^{\ell\to v} \right)^2\left(\lambda^{v\to u}_* \right)^2\omega_\ell^{(2)}, \\
        s_{uv} &= \sum_{\ell\in A_u} \left(\lambda^{\ell\to u} \right)^2\lambda^{u\to v}_* \omega_\ell^{(2)} + \sum_{\ell\in A_v} \left(\lambda^{\ell\to v} \right)^2\lambda^{v\to u}_* \omega_\ell^{(2)} \text{ and } \\
        s_{vv} &= \sum_{\ell\in A_u} \left(\lambda^{\ell\to u} \right)^2\left(\lambda^{u\to v}_*\right)^2\omega_\ell^{(2)} + \sum_{\ell\in A_v} \left(\lambda^{\ell\to v} \right)^2\omega_\ell^{(2)}.
    \end{align*}

    Similarly,
    \begin{align*}
        t_{uuu} &=  \sum_{\ell\in A_u} \left(\lambda^{\ell\to u} \right)^3\omega_\ell^{(3)} + \sum_{\ell\in A_v} \left(\lambda^{\ell\to v} \right)^3\left(\lambda^{v\to u}_* \right)^3 \omega_\ell^{(3)}, \\
        t_{uuv} &= \sum_{\ell\in A_u} \left(\lambda^{\ell\to u} \right)^3\lambda^{u\to v}_* \omega_\ell^{(3)} + \sum_{\ell\in A_v} \left(\lambda^{\ell\to v} \right)^3\left(\lambda^{v\to u}_*\right)^2 \omega_\ell^{(3)}, \\
        t_{uvv} &= \sum_{\ell\in A_u} \left(\lambda^{\ell\to u} \right)^3\left(\lambda^{u\to v}_*\right)^2\omega_\ell^{(3)} + \sum_{\ell\in A_v} \left(\lambda^{\ell\to v} \right)^3 \lambda^{v\to u}_* \omega_\ell^{(3)} \text{ and } \\
        t_{vvv} &= \sum_{\ell\in A_u} \left(\lambda^{\ell\to u} \right)^3\left(\lambda^{u\to v}_*\right)^3\omega_\ell^{(3)} + \sum_{\ell\in A_v} \left(\lambda^{\ell\to v} \right)^3 \omega_\ell^{(3)}.
    \end{align*}
    
    For brevity, we define  $L_u^{(2)} := \sum_{\ell\in A_u} \left(\lambda^{\ell\to u} \right)^2\omega_\ell^{(2)}$, $L_v^{(2)} := \sum_{\ell\in A_v} \left(\lambda^{\ell\to v} \right)^2\omega_\ell^{(2)}$, $L_u^{(3)} := \sum_{\ell\in A_u} \left(\lambda^{\ell\to u} \right)^3\omega_\ell^{(3)}$ and $L_v^{(3)} := \sum_{\ell\in A_v} \left(\lambda^{\ell\to v} \right)^3\omega_\ell^{(3)}$. Then the matrix
    \begin{align}\label{eq:A2}
    \begin{pmatrix*}
       s_{uu} & s_{uv} & s_{vv} \\
       t_{uuu} & t_{uuv} & t_{uvv} \\
       t_{uuv} & t_{uvv} & t_{vvv} 
     \end{pmatrix*} = 
     &\begin{pmatrix*}[r]
      L_{u}^{(2)} & \lambda_{*}^{u\to v} L_{u}^{(2)} & \left(\lambda_{*}^{u\to v}\right)^2L_{u}^{(2)} \\
      L_{u}^{(3)} & \lambda_{*}^{u\to v} L_{u}^{(3)} & \left(\lambda_{*}^{u\to v}\right)^2L_{u}^{(3)} \\
      \lambda_{*}^{u\to v}L_{u}^{(3)} & \left(\lambda_{*}^{u\to v}\right)^2 L_{u}^{(3)} & \left(\lambda_{*}^{u\to v}\right)^3L_{u}^{(3)} 
    \end{pmatrix*} + \\
    &\begin{pmatrix*}[r]
        \left(\lambda_{*}^{v\to u}\right)^2 L_{v}^{(2)} & \lambda_{*}^{v\to u} L_{v}^{(2)} & L_{v}^{(2)} \\
        \left(\lambda_{*}^{v\to u}\right)^3 L_{v}^{(2)} & \left(\lambda_{*}^{v\to u}\right)^2 L_{v}^{(2)} & \lambda_{*}^{v\to u} L_{v}^{(2)} \\
        \left(\lambda_{*}^{v\to u}\right)^2 L_{v}^{(3)} & \lambda_{*}^{v\to u} L_{v}^{(3)} & L_{v}^{(3)}
    \end{pmatrix*}
    \end{align}
 has rank at most $2$ since it is the sum of two rank $1$ matrices. Therefore $d_{uv}^{3\times 3} = 0$.

\end{proof}
\bigskip

We now proceed with the proof of Theorem \ref{d2d3}.

\proofof{Theorem \ref{d2d3}}
We begin by proving the converse of (b): assume there is no trek {simple} $(P_u,P_v)$ with {non-empty sides.} 
Without loss of generality, assume that there is a path from $u$ to $v$. If there also exists a path from $v$ to $u$ then $u$ and $v$ belong to a cycle and then $d_{uv}^{3\times 3} = 0$ by Lemma \ref{lem:d3cycle}.

Therefore, we assume there is no path from $v$ to $u$. In this case, 
\begin{equation*}
    \lambda^{\ell\to v} = 
    \begin{cases}
        \lambda^{\ell\to u}\lambda^{u\to v} & \text{if } \ell\in \an(u)\\
        \lambda^{\ell\to v} & \text{if } \ell\not\in \an(u).
    \end{cases}
\end{equation*}

Following the same notation as in \eqref{eq:sij} and \eqref{eq:a_ell} we have that
\begin{align*}
  a_\ell^{(uu)} =& \left(\lambda^{\ell\to u}\right)^2, \quad
  a_\ell^{(uv)} =  \left(\lambda^{\ell\to u}\right)^2\lambda^{u\to v} = \lambda^{u\to v}a_\ell^{(uu)} \text{ and}\\
  a_\ell^{(vv)}=& \left(\lambda^{\ell\to v}\right)^2 = 
    \begin{cases}
        \left(\lambda^{u\to v}\right)^2 a_\ell^{(uu)} & \text{if } \ell\in \an(u)\\
        \left(\lambda^{\ell\to v}\right)^2 & \text{if } \ell\not\in \an(u) .
    \end{cases} 
\end{align*}
Therefore,
$$ s_{uv} = \lambda^{u\to v}s_{uu}, \quad \text{and} \quad s_{vv} = \left(\lambda^{u\to v}\right)^2s_{uu} + \sum_{\substack{\ell\in V\\ \ell\not\in \an(u)}} \left(\lambda^{\ell\to v}\right)^2\omega_{\ell}^{(2)}.$$
Similarly, it is straightforward to check that 
$$t_{uuv} = \lambda^{u\to v}t_{uuu}, \quad t_{uvv} = \left(\lambda^{u\to v}\right)^2t_{uuu}, \quad \text{and} \quad t_{vvv} = \left(\lambda^{u\to v}\right)^3t_{uuu} + \sum_{\substack{\ell\in V\\ \ell\not\in \an(u)}} \left(\lambda^{\ell\to v}\right)^3\omega_{\ell}^{(3)}.$$
Therefore, the matrix
$$\begin{pmatrix} s_{uu}&s_{uv}&s_{vv}\\t_{uuu}&t_{uuv}&t_{uvv}\\t_{uuv}&t_{uvv}&t_{vvv}\end{pmatrix} = 
\begin{pmatrix} 
s_{uu} & \lambda^{u\to v}s_{uu} & \left(\lambda^{u\to v}\right)^2s_{uu} \\ 
t_{uuu} & \lambda^{u\to v}t_{uuu} & \left(\lambda^{u\to v}\right)^2t_{uuu} \\ 
t_{uuv} & \lambda^{u\to v}t_{uuv} & \left(\lambda^{u\to v}\right)^2t_{uuv}\end{pmatrix} +
\begin{pmatrix} 
0 & 0 & \sum_{\substack{\ell\in V\\ \ell\not\in \an(u)}} \left(\lambda^{\ell\to v}\right)^2\omega_{\ell}^{(2)} \\
0 & 0 & 0 \\
0 & 0 & \sum_{\substack{\ell\in V\\ \ell\not\in \an(u)}} \left(\lambda^{\ell\to v}\right)^3\omega_{\ell}^{(3)}
\end{pmatrix}
$$
has rank $2$ as it is the sum of two rank $1$ matrices, and consequently $d_{uv}^{3\times 3}=0.$ 

The converse of (a) follows directly. Given that there exists no $v\to u$-path and $u$ and $v$ share no common ancestor, the $2\times2$ submatrix 
$$\begin{pmatrix} s_{uu}& s_{uv}\\
t_{uuu}&t_{uuv}\end{pmatrix} = 
\begin{pmatrix} s_{uu} & \lambda^{u\to v}s_{uu} \\ t_{uuu} & \lambda^{u\to v}t_{uuu} \end{pmatrix}
$$ has rank $1$. Therefore $d_{uv}^{2\times 2}=0.$


We will now prove the direct implications, beginning with (b) and then proceeding to (a). First, assume there exists a {simple} trek $\tau=(P_1,P_2)\in\cT(u,v)$ with $\Top(\tau)=\ell$ such that $P_1$ and $P_2$ are {non-empty} 
Furthermore, assume that $P_1$ and $P_2$ are the shortest disjoint paths from $\ell$ to $u$ and $v$, respectively, containing only distinct vertices.
Let $P_1$ consist of vertices $\ell = u_0, u_1,\ldots, u_k = u$ with edges $u_i\to u_{i+1}$, $i=1,\ldots, k-1$ and $P_2$ of vertices $\ell = v_0, v_1,\ldots, v_s = v$ with edges $v_i\to v_{i+1}$, $i=1,\ldots, s-1$.
Set $\lambda_{v_0,v_1} = \lambda_{u_0,u_1}=2$, $\lambda_{u_i,u_{i+1}} = \lambda_{v_j,v_{j+1}} = 1$ for $i=1,\ldots, k-1$ and $j=1,\ldots, s-1$ and set all other $\lambda$'s to 0. Also, set $\omega^{(2)}_{u_i} = \omega^{(3)}_{u_i} = 1$ and 
$\omega^{(2)}_{v_j} = \omega^{(3)}_{v_j} = 1$ for $i=0,\ldots, k$ and $j=0,\ldots, s$.
Moreover, define $S_u^{(2)} := \sum_{i=1}^k \left(\lambda^{u_i\to u}\right)^2\omega^{(2)}_{u_i}$, $S_u^{(3)} := \sum_{i=1}^k \left(\lambda^{u_i\to u}\right)^3\omega^{(3)}_{u_i}$,  $S_v^{(2)} := \sum_{i=1}^k \left(\lambda^{v_i\to v}\right)^2\omega^{(2)}_{v_i}$ and $S_v^{(3)} := \sum_{i=1}^k \left(\lambda^{v_i\to v}\right)^3\omega^{(3)}_{v_i}$, then
\begin{align*}
    d_{uv}^{3\times 3}&=
    \det\begin{pmatrix}
        \left(\lambda^{\ell\to u}\right)^2\omega^{(2)}_{\ell} + S_u^{(2)} & \lambda^{\ell\to u}\lambda^{\ell\to v}\omega^{(2)}_{\ell} & \left(\lambda^{\ell\to v}\right)^2\omega^{(2)}_{\ell} + S_v^{(2)}\\ 
        \left(\lambda^{\ell\to u}\right)^3\omega^{(3)}_{\ell} + S_u^{(3)} & \left(\lambda^{\ell\to u}\right)^2\lambda^{\ell\to v}\omega^{(3)}_{\ell} & \lambda^{\ell\to u}\left(\lambda^{\ell\to v}\right)^2\omega^{(3)}_{\ell} \\
        \left(\lambda^{\ell\to u}\right)^2\lambda^{\ell\to v}\omega^{(3)}_{\ell} & \lambda^{\ell\to u}\left(\lambda^{\ell\to v}\right)^2\omega^{(3)}_{\ell} & \left(\lambda^{\ell\to v}\right)^3\omega^{(3)}_{\ell} + S_v^{(3)}
    \end{pmatrix} \\ &= \det\begin{pmatrix}
        2^2 + k & 2^2 & 2^2 + k\\
        2^3 + k & 2^3 & 2^3\\
        2^3 & 2^3 & 2^3 + k 
    \end{pmatrix} = 12k^2 >0.
\end{align*}

Moreover,
$$d_{uv}^{2\times 2}=
\det\begin{pmatrix}
2^2 + k & 2^2\\
2^3 + k & 2^3
\end{pmatrix}=2^2k >0.$$

Therefore, $d^{2\times 2}_{uv}$ and $d^{3\times 3}_{uv}$ are not zero for all values of $\lambda$ and $\omega$, leading to a contradiction. Thus, there exist no {simple} trek $\tau=(P_1,P_2)\in\cT(u,v)$ with $\Top(\tau)=\ell$ such that $P_1$ and $P_2$ are {non/empty}. 

Alternatively, if there is a $v\to u$ path, then $P_1$ consist of vertices $v = u_0, u_1,\ldots, u_k = u$ with edges $u_i\to u_{i+1}$, $i=1,\ldots, k-1$ and $P_2$ is just the vertex $v$. Using the same parameters and notation as above, we also have that
$$d_{uv}^{2\times 2}=
\det\begin{pmatrix}
\left(\lambda^{v\to u}\right)^2\omega^{(2)}_{v} + S_u^{(2)} & \lambda^{v\to u}\lambda^{v\to v}\omega^{(2)}_{v}\\ 
\left(\lambda^{v\to u}\right)^3\omega^{(3)}_{v} + S_u^{(3)} & \left(\lambda^{v\to u}\right)^2\lambda^{v\to v}\omega^{(3)}_{v}
\end{pmatrix} = 
\det\begin{pmatrix}
2^2 + k & 2^2\\
2^3 + k & 2^3
\end{pmatrix}=2^2k >0,$$
which concludes the proof.

\subsection{Proof of Proposition~\ref{prop:regression}}

The following lemma, necessary for proving Proposition~\ref{prop:regression}, provides a closed-form expression for regression coefficients when a subset of nodes is regressed onto an ancestral set that contains all of its ancestors.

\begin{lem}\label{lem:refCoeffs}
    Let $G=(V,E)$ be a simple cyclic graph with disjoint subsets of nodes $C, D\subseteq V$ such that $C$ is an ancestral set with $\an(D)\subseteq C\cup D$.
    Let $R_{D,C} = S_{D,C}S_{C,C}^{-1}$ be the regressing coefficients when $D$ is regressed onto $C$. Then $$R_{D,C} = (I - \Lambda_{D, D})^{-T}(\Lambda_{C, D})^T.$$
\end{lem}
\begin{proof}

Since $C$ is ancestral, $C$ and $D$ are disjoints and $\an(D)\subseteq C\cup D$, there are no edges from $D$ to $C$ nor from $V\setminus(C\cup D)$ to $C$ or $D$. Therefore, the matrix $\Lambda$ can be written in block form as
    $$\Lambda = 
    \begin{pmatrix}
    \Lambda_{C,C} & \Lambda_{C, D} & \Lambda_{C, V\setminus (C\cup D)}\\
    0 & \Lambda_{D,D} & \Lambda_{D, V\setminus (C\cup D)} \\
    0 & 0 & \Lambda_{V\setminus (C\cup D), V\setminus (C\cup D)}
    \end{pmatrix},$$ where $V\setminus (C\cup D)$ can be empty.
    Hence, the corresponding path matrix $(I-\Lambda)^{-1}$ has the form
    $$(I-\Lambda)^{-1}=\begin{pmatrix}
    (I-\Lambda_{C,C})^{-1} & M_1 & M_2\\
    0 & (I - \Lambda_{D,D})^{-1} & M_3 \\
    0 & 0 & (I - \Lambda_{V\setminus (C\cup D), V\setminus (C\cup D)})^{-1} 
    \end{pmatrix}.$$
   The regressing coefficients when $D$ is regressed onto $C$ are defined as $R_{D,C} = S_{D,C}S_{C,C}^{-1}$, where 
   \begin{align*}
       S_{D,C} &= \left((I-\Lambda)^{-T}\right)_{D, V}\Omega^{(2)}\left((I-\Lambda)^{-1}\right)_{V,C}= \begin{pmatrix} M_1^T & (I-\Lambda_{D,D})^{-T} & 0 \end{pmatrix}\Omega^{(2)} \begin{pmatrix} (I-\Lambda_{C,C})^{-1}\\ 0 \\ 0\end{pmatrix} \\ 
       &=  M_1^T \Omega^{(2)}_{C,C} (I-\Lambda_{C,C})^{-1},
   \end{align*}
    and 
    \begin{align*}
        S_{C,C} &= \left((I-\Lambda)^{-T}\right)_{C, V}\Omega^{(2)}\left((I-\Lambda)^{-1}\right)_{V,C}= \begin{pmatrix} (I-\Lambda_{C,C})^{-T} & 0 & 0 \end{pmatrix}\Omega^{(2)} \begin{pmatrix} (I-\Lambda_{C,C})^{-1}\\ 0 \\ 0\end{pmatrix}\\ 
       &=  (I-\Lambda_{C,C})^{-T}\Omega^{(2)}_{C,C} (I-\Lambda_{C,C})^{-1}.
   \end{align*}
   Therefore, we obtain that
   \begin{align*}
      R_{D,C} &=  M_1^T (I-\Lambda_{C,C})^T.
   \end{align*}
   Our claim that $R_{D,C}=(I - \Lambda_{D,D})^{-T}(\Lambda_{C, D})^T$  now follows because
$$M_1^T = (I - \Lambda_{D, D})^{-T}(\Lambda_{C, D})^T(I-\Lambda_{C, C})^{-T},$$
    which can be deduced from
    \begin{align*}
    0=
    \left[(I-\Lambda)(I-\Lambda)^{-1}\right]_{C,D}
    =(I-\Lambda_{C,C})M_1 - \Lambda_{C, D}(I - \Lambda_{D, D})^{-1}.
    \end{align*} 
\end{proof}

\proofof{Proposition~\ref{prop:regression}}
Observe that Lemma \ref{lem:refCoeffs} applies here, thus, $R_{V\setminus C,C}=(I - \Lambda_{V\setminus C, V\setminus C})^{-T}(\Lambda_{C, V\setminus C})^T $.  Therefore, 
\begin{align*}
  X_{V\setminus C.C} &= X_{V\setminus C} - R_{V\setminus C, C}X_C = X_{V\setminus C} - (I - \Lambda_{V\setminus C, V\setminus C})^{-T}(\Lambda_{C, V\setminus C})^T X_C.
\end{align*}
Since there are no edges from $V\setminus C$ to $C$, 
$$\begin{pmatrix}\varepsilon_C\\\varepsilon_{V\setminus C}\end{pmatrix} = (I - \Lambda^T)X = \begin{pmatrix}I - \Lambda_{C,C}^T& 0\\
-\Lambda_{C, V\setminus C}^T & I - \Lambda_{V\setminus C, V\setminus C}
^T\end{pmatrix}\begin{pmatrix}X_C\\X_{V\setminus C}
\end{pmatrix} $$
$$= \begin{pmatrix}(I - \Lambda_{C,C}^T)X_C\\ -\Lambda_{C,V\setminus C}^TX_C + X_{V\setminus C} - \Lambda_{V\setminus C, V\setminus C}^TX_{V\setminus C}
\end{pmatrix}.
$$
Hence,
\begin{align*}
    X_{V\setminus C.C} & = X_{V\setminus C} - (I - \Lambda_{V\setminus C, V\setminus C})^{-T}(\Lambda_{C, V\setminus C})^T X_C \\
    &= (I - \Lambda_{V\setminus C, V\setminus C})^{-T}\varepsilon_{V\setminus C}.
\end{align*}

\subsection{Proof of Lemma~\ref{lem:rootCycle}}
\begin{proof}
Let $\cC$ be the set containing all root cycles in $G$ constructed as described in Corollary \ref{cor:rootCycles}. Let $C,D\in\cC$ be two distinct sets. 
If $C$ is a root cycle, then it is an ancestral set, and thus $X_C=(I-\Lambda_{C,C})^{-T}\varepsilon_C$.
By Proposition \ref{prop:regression}, we have that $X_{D.C} = (X_{V\setminus C.C})_D = ((I - \Lambda_{V\setminus C, V\setminus C})^{-T}\varepsilon_{V\setminus C})_D$ depends only on $\varepsilon_{V\setminus C}$. This implies that $X_C$ and $X_{D.C}$ are independent.
Then, for all $c\in C$ and $d\in D$, we have that
$$
0 = \EE[X_c^2(X_{D.C})_d]=\EE[X_c^2(X_d-R_{d,C}X_C)].
$$


Conversely, suppose $C$ is not a root cycle. Then there must be a cycle $D\in\cC$ such that there is a directed path from a node $d\in D$ to a node $c\in C$. Take the shortest possible path from $d$ to $c$ so that no interior node of the path is in $C$.

Specialize all weights $\lambda_{uv}$ to zero except for those on the path.  
Then $(R_{D,C})_{dc}=\frac{s_{dc}}{s_{cc}}$ and
$$\EE[X_c^2(X_{D.C})_d]=\EE[X_c^2(X_d-R_{d,c}X_c)] = \EE[X_c^2X_d] - R_{d,C} \EE[X_c^3] =t_{ccd}-\frac{s_{dc}}{s_{cc}}t_{ccc}=0$$ if and only if $ s_{cc}t_{ccd}-s_{dc}t_{ccc} = 0$ or equivalently, $d^{2\times 2}_{cd}=0$.  However, by Theorem \ref{d2d3}(a), $d^{2\times 2}_{cd}$ is not identically zero if $d$ is an ancestor of $c$.
\end{proof}

\subsection{Proof of Lemma~\ref{lem:rankAs}}
\begin{proof}
    Let $C$ be a root cycle and let $u,\ v,\ w\in C$ be three consecutive nodes. Since $C$ is a root cycle and there is an edge $u\to v$, the sets $A_u$ and $A_v$ defined in \eqref{eq:AuAv} are $A_u=C\setminus v$ and $A_v=\{v\}$. Therefore, following a similar construction as in Lemma \ref{lem:d3cycle} we obtain
    \begin{align*}
    A_{uv}^{(2)} =
    \begin{pmatrix*}
       1 & \lambda_{uv} & \lambda_{uv}^2 \\
       s_{uu} & s_{uv} & s_{vv} \\
       t_{uuu} & t_{uuv} & t_{uvv} \\
       t_{uuv} & t_{uvv} & t_{vvv} 
     \end{pmatrix*} = 
     &\begin{pmatrix*}[r]
     1 & \lambda_{uv} & \lambda_{uv}^2 \\
      L_{u}^{(2)} & \lambda_{uv} L_{u}^{(2)} &\lambda_{uv}^2L_{u}^{(2)} \\
      L_{u}^{(3)} & \lambda_{uv} L_{u}^{(3)} & \lambda_{uv}^2L_{u}^{(3)} \\
      \lambda_{uv}L_{u}^{(3)} & \lambda_{uv}^2 L_{u}^{(3)} & \lambda_{uv}^3L_{u}^{(3)} 
    \end{pmatrix*} + \\
    &\begin{pmatrix*}[r]
    0 & 0 & 0 \\
        \left(\lambda_{*}^{v\to u}\right)^2 L_{v}^{(2)} & \lambda_{*}^{v\to u} L_{v}^{(2)} & L_{v}^{(2)} \\
        \left(\lambda_{*}^{v\to u}\right)^3 L_{v}^{(2)} & \left(\lambda_{*}^{v\to u}\right)^2 L_{v}^{(2)} & \lambda_{*}^{v\to u} L_{v}^{(2)} \\
        \left(\lambda_{*}^{v\to u}\right)^2 L_{v}^{(3)} & \lambda_{*}^{v\to u} L_{v}^{(3)} & L_{v}^{(3)}
    \end{pmatrix*}
    \end{align*}
    where $L_u^{(2)} := \sum_{\ell\in A_u} \left(\lambda^{\ell\to u} \right)^2\omega_\ell^{(2)}$, $L_v^{(2)} := \sum_{\ell\in A_v} \left(\lambda^{\ell\to v} \right)^2\omega_\ell^{(2)}$, $L_u^{(3)} := \sum_{\ell\in A_u} \left(\lambda^{\ell\to u} \right)^3\omega_\ell^{(3)}$ and $L_v^{(3)} := \sum_{\ell\in A_v} \left(\lambda^{\ell\to v} \right)^3\omega_\ell^{(3)}$.
    Consequently, $A_{uv}^{(2)}$ has rank at most $2$ since it is the sum of two rank $1$ matrices.

Moreover, since $\lambda^{v\to w} = \lambda^{v\to v}\lambda_{vw}$ we have
\begin{align*}
    s_{uw} &=  
    \sum_{\ell \in A_u} \lambda^{\ell\to u}\lambda^{\ell\to w}\omega_\ell^{(2)} + \lambda^{v\to u}\lambda^{v\to w}\omega_v^{(2)} 
    = L_{uw\setminus v}^{(2)} + \lambda_*^{v\to u}\lambda_{vw}L_{v}^{(2)} \\
    s_{vw} &= 
    \sum_{\ell \in A_u} (\lambda^{\ell\to u}\lambda_{uv})\lambda^{\ell\to w}\omega_\ell^{(2)} + \lambda^{v\to v}\lambda^{v\to w}\omega_v^{(2)} = 
    \lambda_{uv} L_{uw\setminus v}^{(2)} + \lambda_{vw} L_{v}^{(2)}
\end{align*}
where $L_{uw}^{(2)} := \sum_{\ell\in A_u} \lambda^{\ell\to u}\lambda^{\ell\to w}\omega_\ell^{(2)}$. Defining $L_{uuw}^{(3)} := \sum_{\ell\in A_u} \left(\lambda^{\ell\to u}\right)^2\lambda^{\ell\to w}\omega_\ell^{(2)}$ we also have that
\begin{align*}
     t_{uuw} &= \sum_{\ell \in A_u} \left(\lambda^{\ell\to u}\right)^2\lambda^{\ell\to w}\omega_\ell^{(3)} + \left(\lambda^{v\to u}\right)^2\lambda^{v\to w}\omega_v^{(3)} 
    = L_{uuw}^{(3)} + (\lambda_*^{v\to u})^2\lambda_{vw}L_{v}^{(3)}\\
     t_{uvw} &= \sum_{\ell \in A_u} \lambda^{\ell\to u}\lambda^{\ell\to v}\lambda^{\ell\to w}\omega_\ell^{(3)} + \lambda^{v\to u}\lambda^{v\to v}\lambda^{v\to w}\omega_v^{(3)} 
    = \lambda_{uv}L_{uuw}^{(3)} + \lambda_*^{v\to u}\lambda_{vw}L_{v}^{(3)}\\
     t_{vvw} &= \sum_{\ell \in A_u} \left(\lambda^{\ell\to v}\right)^2\lambda^{\ell\to w}\omega_\ell^{(3)} + \left(\lambda^{v\to v}\right)^2\lambda^{v\to w}\omega_v^{(3)} 
    = \lambda_{uv}^2L_{uuw}^{(3)} + \lambda_{vw}L_{v}^{(3)}.
\end{align*}

 Therefore,
    \begin{align}
    \label{eq:decompA3}
    A_{uv}^{(3)} = &\begin{pmatrix*}[r]
      1 \quad & \lambda_{uv} \quad& \lambda_{uv}^2 \quad & 0 & 0 \\
      L_{u}^{(2)} & \lambda_{uv}L_{u}^{(2)} & \lambda_{uv}^2L_{u}^{(2)} & 0 & 0\\
      L_{u}^{(3)} & \lambda_{uv}L_{u}^{(3)} & \lambda_{uv}^2L_{u}^{(3)} & 0 & 0\\
      \lambda_{uv}L_{u}^{(3)} & \lambda_{uv}^2L_{u}^{(3)} & \lambda_{uv}^3L_{u}^{(3)} & 0 & 0
    \end{pmatrix*} + \begin{pmatrix*}[r]
    0 & 0 & 0 & \lambda_{uv}^2 \quad & \lambda_{uv}^3 \quad \\
    0 & 0 & 0 & L_{uw}^{(2)} & \lambda_{uv}L_{uw}^{(2)}\\
    0 & 0 & 0 & L_{uuw}^{(3)} & \lambda_{uv}L_{uuw}^{(3)}\\
    0 & 0 & 0 & \lambda_{uv}L_{uuw}^{(3)}& \lambda_{uv}^2L_{uuw}^{(3)}\\
    \end{pmatrix*} +  \\
     &\begin{pmatrix*}[r]
      0 \qquad\qquad & 0 \qquad\qquad & 0 \qquad\qquad & 0 \qquad\qquad & 0 \qquad\qquad \\
      (\lambda_*^{v\to u})^2L_{v}^{(2)} & \lambda_*^{v\to u}L_{v}^{(2)} & L_{v}^{(2)} & \lambda_*^{v\to u}\lambda_{vw}L_{v}^{(2)}  &  \lambda_{vw} L_{v}^{(2)}\\
      (\lambda_*^{v\to u})^3L_{v}^{(3)} & (\lambda_*^{v\to u})^2L_{v}^{(3)} & \lambda_*^{v\to u}L_{v}^{(3)}  & (\lambda_*^{v\to u})^2\lambda_{vw}L_{v}^{(3)}  & \lambda_*^{v\to u}\lambda_{vw}L_{v}^{(3)} \\
      (\lambda_*^{v\to u})^2L_{v}^{(3)} & \lambda_*^{v\to u}L_{v}^{(3)} & L_{v}^{(3)}  & \lambda_*^{v\to u}\lambda_{vw}L_{v}^{(3)} & \lambda_{vw}L_{v}^{(3)}
    \end{pmatrix*}.\nonumber
    \end{align}
  Consequently, $A_{uvw}^{(3)}$ has rank at most $3$ since it is the sum of three rank $1$ matrices.

Finally, the matrices $A_{uv}^{(2)}$ and $A_{uvw}^{(3)}$ have ranks $2$ and $3$ respectively, for generic choices of the parameters, provided their determinants are not identically zero. To establish this, it suffices to demonstrate a single set of parameters for which their determinants are nonzero.
 Take $\lambda_{uv} = 2$ and $\lambda_{ij} = 1$ for all other edge $i\to j$ in $C$. Also, set $\omega_i^{(2)} = \omega_i^{(3)} = 1$ for all $i\in C$. Straightforward computations show that, for this choice of parameters, we have
    $$L_{u}^{(2)} = k-1, \quad L_{u}^{(2)} = -k+1, \quad L_{uw}^{(2)} = 2k-3, \quad L_{uuw}^{(2)} = -2k+3,  $$
    $$ L_v^{(2)} = 1, \quad L_v^{(3)} = -1, \quad \frac{\lambda_*^{v\to v}}{\lambda_{uv}} = 1.$$

 Then, the matrix
 $$A_{uv}^{(2)} = \begin{pmatrix*}
    1 & 2 & 4 \\
    k & 2k-1 & 4k-3 \\
    -k & -2k+1 & -4k+3 \\
    -2k+1 & -4k+3 & -8k+7
 \end{pmatrix*} 
 $$ has rank $2$ since all its $3\times 3$ minors vanish and $\det  \begin{pmatrix*} 1 & 2 \\ k & 2k-1 \end{pmatrix*} =-1 \neq 0.$ Similarly,
  $$A_{uv}^{(2)} = \begin{pmatrix*}
    1 & 2 & 4 & 4 & 8\\
    k & 2k-1 & 4k-3 & 2k-2& 4k-5\\
    -k & -2k+1 & -4k+3 & -2k+2 & -4k+5 \\
    -2k+1 & -4k+3 & -8k+7 & -4k+5& -8k+11
 \end{pmatrix*} $$ has rank $3$ since all its $4\times 4$ minors vanish and $\det  \begin{pmatrix*} 1 & 2 & 4\\ k & 2k-1 & 2k-2 \\ -2k+1 &-4k+3 &-4k+5 \end{pmatrix*} =-2k+1 \neq 0 \text{ for } k\geq 1.$
  
\end{proof}

\subsection{Proof of Theorem~\ref{thm:lambda01}}

\begin{proof}
The proof of $(i)$ follows from the decomposition of $A_{uv}^{(3)}$ given in \eqref{eq:decompA3}. 

Consider the submatrix $A$ of $A_{uv}^{(3)}$ formed by taking the last three rows, that is \begin{equation*}
    A := \left(\begin{array}{ccccc}
      s_{uu} & s_{uv}   & s_{vv} & s_{uw}  & s_{vw} \\
      t_{uuu} & t_{uuv} & t_{uvv}& t_{uuw}  & t_{uvw} \\
      t_{uuv} & t_{uvv} & t_{vvv}& t_{uvw}  & t_{vvw}
    \end{array}\right).
  \end{equation*}
  Let $A = M+N+O$  be the decomposition of $A$ as in\eqref{eq:decompA3}. Thus, $M$, $N$ and $O$ are rank-1 matrices formed by the last three rows of the corresponding matrices in \eqref{eq:decompA3}, taken in the same order. Denote by $A_i$ the $i$-th column of $A$ (and similarly for $M$, $N$ and $O$). Define $q(s,t) := \det(A_2\ A_4 \ A_5)$ and $p(s,t) := \det(A_1\ A_4 \ A_5)$. Then
  \begin{align}
       q(s,t) =& \det(A_2\ A_4 \ A_5) = \det(M_2+O_2 \quad N_4+O_4 \quad N_5+O_5)\nonumber \\
       =& \det(M_2 \quad N_4+O_4 \quad N_5+O_5) + \det(O_2 \quad N_4+O_4 \quad N_5+O_5)=\det(M_2 \quad N_4+O_4 \quad N_5+O_5),\label{eq:multDet}
  \end{align}
  where the second to last equality follows from the multilinearity of the determinant with respect to its columns, and the last equality follows because $\det(O_2 \quad N_4+O_4 \quad N_5+O_5)$ is zero as it is the determinant of the sum of two rank-$1$ matrices. 

Similarly,
$$p(s,t) = \det(A_1\ A_4 \ A_5) = \det(M_1 \quad N_4+O_4 \quad N_5+O_5) + \det(O_1 \quad N_4+O_4 \quad N_5+O_5) = \det(M_1 \quad N_4+O_4 \quad N_5+O_5)$$
since $\det(O_1 \quad N_4+O_4 \quad N_5+O_5)$ is zero as it also is the determinant of the sum of two rank-$1$ matrices.

Finally, since $M_2 = \lambda_{uv}M_1$ we have $$ q(s,t) = \det(M_2 \quad N_4+O_4 \quad N_5+O_5) = \lambda_{uv}\det(M_1 \quad N_4+O_4 \quad N_5+O_5) = \lambda_{uv}p(s,t).$$

  Straightforward computations show that the polynomials $q(s,t) = \det(A_2\ A_4 \ A_5)$ and $p(s,t) = \det(A_1\ A_4 \ A_5)$ correspond to the polynomials in the statement of Theorem Theorem~\ref{thm:lambda01}.

To prove $(ii)$, assume $C$ has only $2$ nodes $u, v$. The quadratic equation \eqref{eq:linEqlambda2} on $\lambda_{uv}$  follows from the fact that the $3\time 3$ minors of  $A_{uv}^{(2)}$ vanish. Specifically,  taking the determinant of the first three rows of $A_{uv}^{(2)}$ yields
  $$(s_{uu}t_{uuv}-s_{uv}t_{uuu})\lambda_{uv}^2+(s_{vv}t_{uuu}-s_{uu}t_{uvv})\lambda_{uv}+s_{uv}t_{uvv}-s_{vv}t_{uuv}.$$
  Furthermore, the coefficient of $\lambda_{uv}^2$, $s_{uu}t_{uuv}-s_{uv}t_{uuu}$, is nonzero (see Equation \eqref{eq:A2}). 
\end{proof}

\subsection{Causal Graph Discovery Algorithm}
\subsubsection{Implementation details}\label{sec:algodetails}

We now give details of the implementation of Algorithm~\ref{alg:disc}.

\begin{itemize}
    \item Line 2: To certify if $r$ is a root we test $H_0: d_{ur}^{2x2} = 0 \quad \forall u \neq r$. Specifically, for each $u \neq r$ we calculate a p-value for $H_0: d_{ur}^{2x2} = 0$ which is calibrated using asymptotic distribution of $\hat d_{ur}^{2x2}$ derived via the delta method. We then control for multiple testing across the $p(p-1)$ individual tests.

    \item Line 4: To find all maximal sets $C$ such that $d_{uv}^{3\times 3} = 0, d_{uv}^{2x2} \neq 0,  d_{vu}^{2x2} \neq 0)$ for all $u, v \in C$ we do the following. To test if $d_{uv}^{3\times 3} = 0$ we calculate a p-value for $H_0: d_{uv}^{3 \times 3} = 0$ which is calibrated using the asymptotic normal distribution of $\hat d_{ur}^{3 \times 3}$ using the delta method. We apply a multiple testing correction to these $p(p-1)/2$ tests.  To test if both $d_{uv}^{2x2}$ and  $d_{vu}^{2x2}$ we take the maximum previously calculated p-value (after adjusting for multiple testing) of the two tests $H_0: d_{vu}^{2x2} = 0$ and  $H_0: d_{uv}^{2x2} = 0$. We construct an undirected graph with an edge between each pair in the set $(u,v: d_{uv}^{3\times 3} = 0, d_{uv}^{2x2} \neq 0,  d_{vu}^{2x2} \neq 0)$. We then let $\mathcal{C}$ be the set of maximal cliques in that graph. 
    
    \item In practice, the set \[C^\star = (\{u,v\}: d_{uv}^{3\times 3} = 0, d_{uv}^{2x2} \neq 0,  d_{vu}^{2x2} \neq 0)\] may be empty given a fixed hypothesis testing level. When this is the case, we ensure that $C^\star$ is non-empty with the following two steps. First, if the set $C^2 = (u,v: d_{uv}^{2x2} \neq 0,  d_{vu}^{2x2} \neq 0)$ is empty for our initial hypothesis testing level, we set it to 
    \[C^2 = (u,v) = \arg\min_{v',v''} \max(\pi^2_{v',v''}, \pi^2_{v'',v'}) \]
    where $\pi^2_{v',v''}$ is the p-value when testing if $d_{v',v''}^{2 \times 2} = 0$. Once $C^2$ is non-empty, we then set 
    \[C^3 = (\{u,v\}: d_{uv}^{3\times 3} = 0). \]
    If $C^2 \cap C^3 = \emptyset$, we then lower the hypothesis testing level until it is non empty; i.e., 
     \[C^3 = (u,v) = \arg\max_{(v',v'') \in C^2} \pi^3_{v',v''} \]
     where $\pi^3_{v'',v'}$ is the p-value for testing if $d_{v', v''}^{3 \times 3} = 0$. This ensures that $C^\star = C^2 \cap C^3$ is non-empty.
    \item Line 5: To prune $\mathcal{C}$ to obtain $\mathcal{C}'$ all root cycles we do the following. For each maximal clique $C_1 \in \mathcal{C}$ and let $C_2 = (v : v \in C \text{ for some } C \in  \mathcal{C}) \setminus C_1$. Letting $X_{C_2.C_1}$ be the residuals when regressing $X_{C_2}$ onto $X_{C_1}$, we then test if $E(X_{C_2.C_1} X_{C_1}^2)$ is the zero matrix using empirical likelihood. 
    We then apply a multiple testing procedure to each of the $|\mathcal{C}|$ tests. In practice, all $C \in \mathcal{C}$ may be rejected, in which case we simply set the root cycle to be the union of all maximal cliques. Alternatively, we could pick the maximal clique with the largest p-value, but that ends up doing worse empirically. If, multiple maximal cliques in $C \in \mathcal{C}$ are not rejected which happen to have non-empty intersection, we combine them into a larger single cycle. 

    \item The output of Part 1 is a topological layer where each layer is either a set of singleton nodes or a set of cycles. For each layer with cycles, we construct an undirected cycle for the set $D$ using the greedy approach described in Algorithm~\ref{alg:undirectedCycle}.
\begin{algorithm}
 \caption{\label{alg:undirectedCycle}Find Cycle}
 \begin{algorithmic}[1]
  \Require $\hat S$: the sample correlation of nodes in a $p$-cycle 
  \State Start with an empty undirected graph $G = \{V = [p], E = \emptyset\}$
    \For{$k = 1\ldots p-1$}                    
  \State $E_{\text{valid}} = \{ (u,v): u \text{ and } v \text{ are not connected in G and have degree less than 2}\}$
  \State Add the edge $\arg \max_{(u,v) \in E_{valid}} \vert (\hat S)^{-1}_{u,v}\vert $ to $E$
\EndFor
  \State Add the edge $(u,v)$ to $G$ where $u$ and $v$ are the only remaining nodes with degree $1$
  \State Estimate the path weights when orienting the cycle in both directions
  \State Select the orientation with the smaller absolute value of the product of edgeweights
 \end{algorithmic}
\end{algorithm} 
    \item Once we have established the undirected cycle, we have two sets of edgeweights for each edge in the cycle (via Theorem 4.13) corresponding to the two possible orientations of the cycle. We estimate both directions, then pick the cycle orientation which has the smaller product of the edge weights.
    \item Line 8: Once we have oriented the cycle and estimated the edgeweights of the cycle $\hat \Lambda_{D,D}$, we select parents for each node in the cycle through the following procedure:
    \begin{itemize}
        \item Let $C$ denote all nodes which are in a layer preceeding $D$. We estimate the edgeweights from $C \rightarrow D$ using \[\hat \Lambda_{C,D} = \hat R_{CD}(I - \hat \Lambda_{D, D})\] where $\hat R_{CD}$ is the matrix of least squares coefficients when regressing $D$ onto $C$. 
        \item Then for each $d \in D$ and $c \in C$, we test the existence of edge $c \rightarrow d$ by the following. Let $d' \in D$ be the parent of $d$ in the cycle. Then, we estimate the residuals of $Y_d$ when setting the coefficient $\lambda_{cd} = 0$ by the following:
        \[\hat \epsilon_{d, -c} = Y_d - \sum_{c' \neq c}\hat \lambda_{cd}Y_{c'} - \hat \lambda_{d',d}Y_{d'}. \] We see empirically that failing to account for the estimation of the nuissance parameters $\hat \lambda_{C\setminus c, d}$ and $\lambda_{d',d}$ and naively testing $E(\hat \epsilon_{d, -c} Y_c^2) = 0$ using empirical likelihood does not maintain nominal size. Thus, we instead take the following $|C| +1$ estimating equations.   
        \begin{equation}
            \begin{aligned}
                m_1(X) &= \hat \epsilon_{d, -c} Y_{c}^2 \\
                m_k(X) &= \hat \epsilon_{d, -c} Y_{c'}^2 \text{ for } c' \in C \text{ and } 1 < k\leq |C|
                \\ 
                m_{|C| + 1}(X) &= \hat \epsilon_{d, -c}^2 Y_{c}  \\
            \end{aligned}
        \end{equation}
        Letting $A \in \mathbb{R}^{|C| + 1 \times |C| + a}$ Jacobian of the estimating equations with respect the coefficients. Without loss of generality, we will assume that the potential parent of interest $c$ is first in $C$. Then, we have
        \[g(X) = m_1(X) -  A_{1, -1} A_{-1, -1}^{-1}m_{2:(|C| + 1)}(X).\]
        Using empirical likelihood to test if $g(X) = 0$ for each $(c, d) \in V \times V$ such that $c$ is in a topological layer preceding $d$ yields a p-value $\pi_{c,d}$. To estimate the final edge set, we add edge $c \rightarrow d$ if $\pi_{c,d}$ after being adjusted for multiple testing exceeds a threshold. 
    \end{itemize}
\end{itemize}

\subsubsection{{Limitations and Future Directions}}
{ It is valuable to consider the behavior of the algorithm when its underlying assumptions are violated. This issue is not unique to our method; for instance, DAG-based algorithms like DirectLiNGAM \citep{shimizu2011directlingam} also face challenges when their acyclicity assumption is violated by the presence of cycles or confounding.}

{In our case, a key limitation arises when the graph structure contains overlapping cycles. In such cases, our method can correctly infer the causal structure up to the point where it encounters a root component that is a strongly connected component but is not a simple cycle (i.e., it contains two or more intersecting cycles). At this juncture, two strategies are possible. The first is a {conservative approach} where the algorithm reports that no further simple cycles can be identified and terminates, providing no additional details about the subgraph induced by the remaining nodes. The second is a more {aggressive strategy} that attempts to find a “best-fitting simple cycle model” by, for example, identifying a cycle that minimizes the violation of the required polynomial equalities.}

{Our current implementation follows the conservative strategy. When no simple root cycle is found, our algebraic conditions fail the statistical tests, and the algorithm correctly indicates the absence of a simple cycle before halting. While the more aggressive strategy is common in implementations of related algorithms like DirectLiNGAM, we have not yet implemented it for our method. Developing and evaluating such a strategy presents an interesting and potentially fruitful direction for future work.
}

\section{Additional Numerical Experiments}
In this section we first present additional numerical experiments. Specifically, we generate data in the same way as described in Section~\ref{sec:numerical_experiments}. However, we now create graphs with $p/5$ disjoint cycles where each cycle is of size $5$. In general, the procedure performs better when the graphs have $3$-cycles when compared to the setting with $5$-cycles. Similar to the setting where the graphs consisted of $3$-cycles, the procedure performs better when the errors are a mixture of normals when compared to the gamma errors.

\begin{figure}[h]
\centering
\includegraphics[scale = .4]{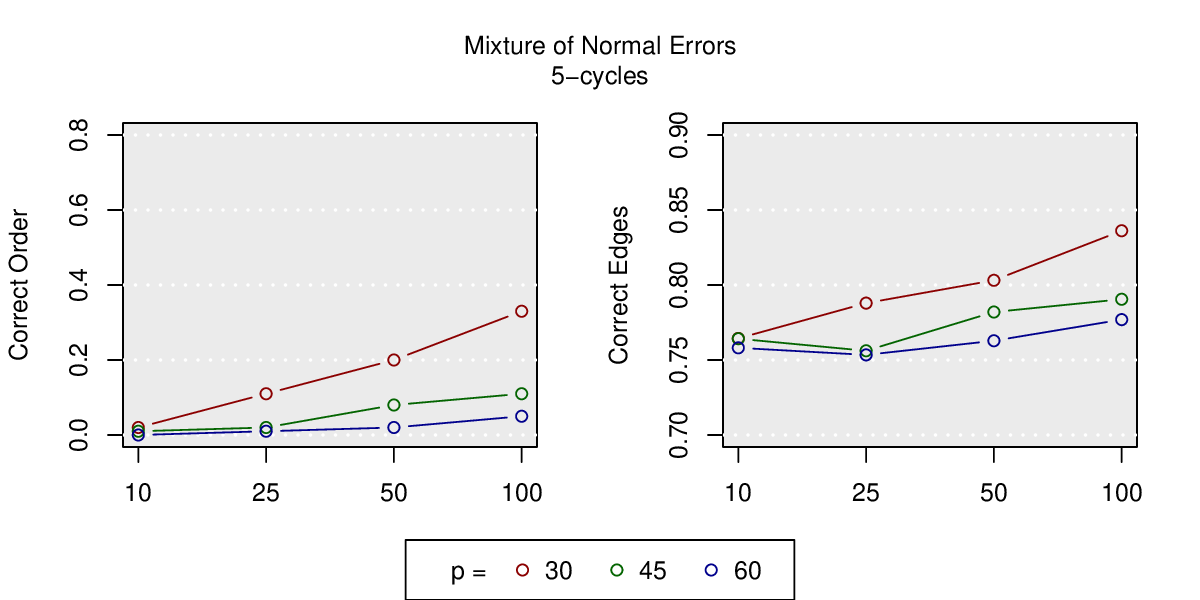}
\includegraphics[scale = .4]{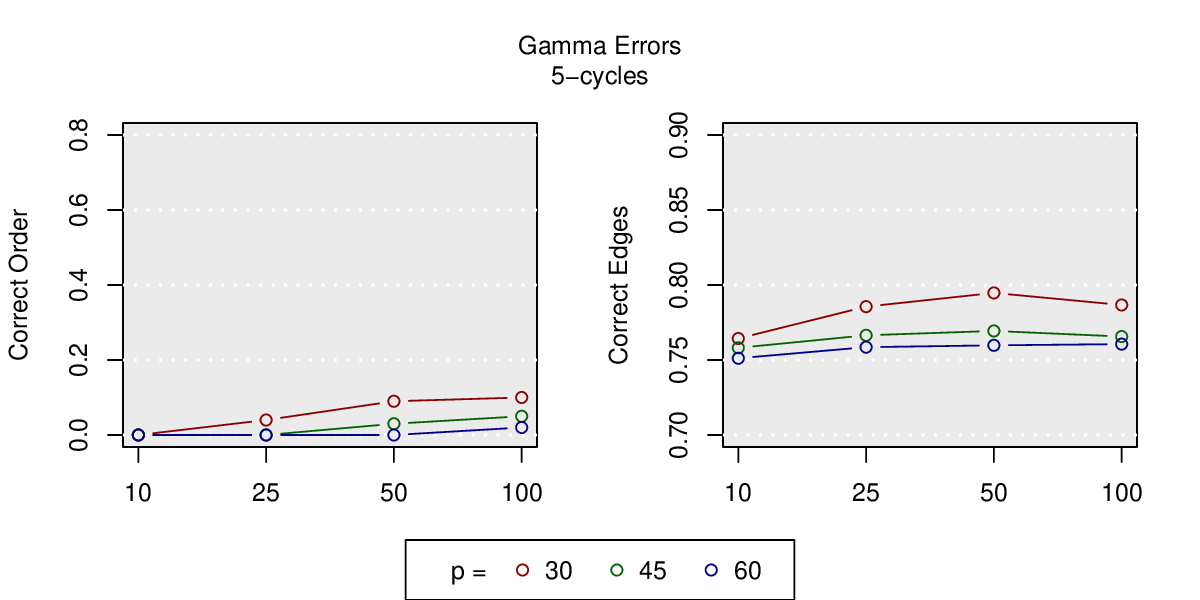}

\caption{\label{fig:5cycles}Results when each graph consists of $5$-cycles.}
\end{figure}

In Table~\ref{tab:ica}, we give details for the ICA based procedure of~\citet{DBLP:conf/uai/LacerdaSRH08}. Specifically, we generate random graphs with $3$-cycles and data as described in Section~\ref{sec:numerical_experiments}. We use the implementation from the rpy-tetrad~\citep{ramsey2023py} and run the experiments on a cluster with 24GB of memory allocated to each job. Table~\ref{tab:ica} shows the number of replicates (out of 25 for each setting) where the ICA based procedure returns an estimate within 24 hours. In general, we see that the computational performance improves as $n$ grows larger and that performance is better for the mixture of normals when compared to the gamma errors. This is because when the demixing matrix is estimated more precisely fewer permutations must be considered in the post-processing procedures. We note that in each instance where an estimated graph is returned, the true graph is recovered exactly.

\begin{table}[ht]
\centering
\caption{\label{tab:ica}The number of replicates where the ICA approach returns an estimated graph within 24 hours.} 
\begin{tabular}{|c|c|c|c|}
  \hline
 & & \multicolumn{2}{|c|}{Completed replicates out of 25} \\ \hline
p & n (in 000s) & Mixture of normals & Gamma  \\ 
  \hline
 \multirow{4}{*}{18} & 10 &  16 & 1\\ 
    & 25 &  21 & 4 \\ 
    & 500 &  24 & 9\\ 
    & 100 &  24 & 14\\ \hline
   \multirow{4}{*}{21} & 10 &   4 & 0\\ 
    & 25 &   9 & 0 \\ 
    & 50 &  16 & 1\\ 
    & 100 &  22 & 10 \\ \hline
\multirow{4}{*}{24} & 10 &   4 & 0\\ 
    & 25 &   9 & 0 \\ 
    & 50 &  1 & 0\\ 
    & 100 &  1 & 0\\ \hline
   \multirow{4}{*}{27} & 10 &   0 & 0 \\ 
    & 25 &   0 & 0\\ 
    & 50 &  0 & 0\\ 
    & 100 &  0 & 0\\
   \hline
\end{tabular}
\end{table}

\end{document}